\documentclass[12pt, reqno]{amsart}
\usepackage{amsmath, amsthm, amscd, amsfonts, amssymb, graphicx, color, mathrsfs}
\usepackage[bookmarksnumbered, colorlinks, plainpages]{hyperref}
\usepackage[all]{xy} 
\usepackage{slashed}
\usepackage{cite}
\usepackage{xcolor}
\usepackage{listings}
\usepackage{xcolor}

\newtheorem{theorem}{Theorem}[section]

\newtheorem{corollary}[theorem]{Corollary}
\theoremstyle{definition}
\newtheorem{definition}[theorem]{Definition}

\newtheorem{conjecture}[theorem]{Conjecture}

\theoremstyle{remark}
\newtheorem{remark}[theorem]{Remark}
\numberwithin{equation}{section}

\begin{document}
\setcounter{page}{1}

\color{darkgray}{
\noindent \centering
{\small   }\hfill    {\small }\\
{\small }\hfill  {\small }}

\centerline{}

\centerline{}


\title[Local smoothing estimates for bilinear FIOs]{Local smoothing estimates for bilinear Fourier integral operators}

\author[D. Cardona]{Duv\'an Cardona$^{1}$}
\address{
  Duv\'an Cardona:
  \endgraf
  Department of Mathematics: Analysis, Logic and Discrete Mathematics
  \endgraf
  Ghent University,
  \endgraf
  Ghent-Belgium.
  \endgraf
  {\it E-mail address:} {\rm duvanc306@gmail.com, duvan.cardonasanchez@ugent.be.}
  \endgraf
 Department of Mathematics
  \endgraf
  Pontificia Universidad Javeriana, 
  \endgraf
  Bogot\'a-Colombia.
  \endgraf
  {\it E-mail address:} {\rm cardonaduvan01@javeriana.edu.co}  
 \endgraf
  Current affiliation:
 \endgraf
  King Fahd University of Petroleum and Minerals 
   \endgraf
    Dhahran,  Saudi Arabia.
  }

\thanks{{$^{1}$ {Website:}
\url{https://sites.google.com/site/duvancardonas/home}}
\newline  Duv\'an Cardona$^{1}$  has been  supported  by the Department of Mathematics of the King Fahd University of Petroleum and Minerals (KFUPM). He also has been supported by the FWO  Odysseus  1  grant  G.0H94.18N:  Analysis  and  Partial Differential Equations and by the Methusalem programme of the Ghent University Special Research Fund (BOF)
(Grant number 01M01021), by the FWO Fellowship
Grant No 1204824N and by the FWO Grant K183725N of the Belgian Research Foundation FWO. He also has been supported by the Oberwolfach Leibniz Fellow of the Mathematical Institute of Research of Oberwolfach, MFO-Germany, Project F2511 (2026) and by the Department of Mathematics of the Pontificia Universidad Javeriana, Bogot\'a-Colombia.
\newline $^{*}$\url{Orcid: 0000-0002-8148-2624}}

\begin{abstract} We formulate a local smoothing conjecture for bilinear Fourier integral operators in every dimension $d \ge 2,$ derived from the celebrated linear case due to Sogge, which we refer to as the \emph{bilinear smoothing conjecture}.  We show that the linear local smoothing conjecture implies this bilinear version. As a consequence of our approach and due to the recent progress on the subject, we establish local smoothing estimates for bilinear Fourier integral operators in dimension $d=2,$ that is, on $\mathbb{R}^2_x \times \mathbb{R}_t$. Also, a partial progress is presented for the high-dimensional case $d\geq 3.$ In particular, our method allows us to deduce that the bilinear local smoothing conjecture holds for all odd dimensions $d$.
\end{abstract} \maketitle
\allowdisplaybreaks
\tableofcontents


\section{Introduction}
This note proposes a conjecture on local smoothing estimates for bilinear Fourier integral operators (FIOs) derived from the linear case observed by Sogge \cite{Sogge1991}. Recent progress in the field (see \cite{BeltranHickmanSogge,GaoLiuMiaoXi2023}) allows our methods to fully prove the conjecture for $d=2$ and to advance the problem partially for $d \geq 3$. In particular, we deduce that the bilinear local smoothing conjecture is true when $d$ is odd. We first recall the linear local smoothing conjecture \cite{Sogge1991} and then introduce its bilinear form. Our main result is Theorem \ref{main:theorem:2}; our approach investigates the smoothing properties of high and of low frequencies of a bilinear Fourier integral operator. For the analysis of the high frequency portion, among other things, we make use of a boundedness theorem for maximal functions proved by Bourgain in \cite{Bourgain1985}.

According to H\"ormander \cite{Hormander1971Ac}, a Fourier integral operator $T$  of order $m\in \mathbb{R}$ and with real phase function $\phi,$   is a continuous linear operator $T:C^\infty_0(\mathbb{R}^d)\rightarrow C^\infty(\mathbb{R}^d\times \mathbb{R}) $  of  the form
\begin{equation}\label{FIO:Quantisation}    T_a^\phi f(x,t)=\smallint\limits_{\mathbb{R}^d}e^{i\phi(x,t,\xi)}a(x,t,\xi)\widehat{f}(\xi)d\xi,
\end{equation} where $\widehat{f}$ is the Fourier transform of $f\in C^\infty_0(\mathbb{R}^d).$ The symbol $a:=a(x,t,\xi)$ is of Kohn-Nirenberg order $m\in \mathbb{R},$ that is $a$ satisfies the symbol inequalities
\begin{equation}
    \forall\alpha\in \mathbb{N}_0,\,\forall\beta \in \mathbb{N}_0,\,\,|\partial_{x}^\beta\partial_\xi^\alpha a(x,t,\xi)|\leq C_{\alpha,\beta}(1+|\xi|)^{m-|\alpha|}.
\end{equation} The symbol $a$ has compact support in $(x,t).$ We assume that in an open neighbourhood $V$ of $\pi(\textnormal{supp}(a)):=\{(x,t):(x,t,\xi)\in \textnormal{supp}(a)\},$ (with $\overline{V}$ being a compact set) the real-valued phase function $\phi$ is homogeneous of degree $1$ in $\xi\neq 0,$ smooth in $(x,t,\xi)\in {V}\times (\mathbb{R}^d\setminus \{0\}),$  and satisfies the non-degeneracy condition
\begin{equation}
  \forall (x,t,\xi)\in \overline{V}\times (\mathbb{R}^d\setminus \{0\}),\,  \textnormal{det}\left((\partial_{x,\xi}^2\phi(x,t,\eta))\right)_{1\leq i,j\leq d }\neq 0.
\end{equation}  
The corresponding $L^p$-theory for these operators can be summarised as follows, see Tao \cite{Tao} and Seeger, Sogge and Stein \cite{SSS}.
\begin{theorem}\label{Tao:SSS:theorem} Let $t\in \mathbb{R}$ be fixed in the support of the symbol $a:=a(x,t,\xi)$. Let $T=T_{a(\cdot,t,\cdot)}^{\phi(\cdot,t,\cdot)}$ be a Fourier integral operator of order $m\in \mathbb{R}$. Then:
\begin{itemize}
    \item[1.] If $m=-(d-1)/2,$  $T$ is of weak $(1,1)$ type:
    \begin{equation}\label{Tao:t}
       \exists C_t>0,\forall f\in L^1(\mathbb{R}^d),\,\,\Vert T f\Vert_{L^{1,\infty}(\mathbb{R}^d_x)}\leq C_t\Vert f \Vert_{L^1(\mathbb{R}^d_x)}.
    \end{equation}
    \item[2.] Let $1<p<\infty.$ Let $T$ be a Fourier integral operator of order $m=-(d-1)|\frac{1}{2}-\frac{1}{p}|.$  Then $T$ is bounded from $L^p(\mathbb{R}^d_x)$ to $L^p(\mathbb{R}^d_x).$ When $p=1,$ then $T$ is bounded from the Hardy space $H^1(\mathbb{R}^d_x)$ to $L^1(\mathbb{R}^d_x).$ As in \eqref{Tao:t}, these operator norms depend on $t.$
\end{itemize}
\end{theorem}
The case $p=2$ in Theorem \ref{Tao:SSS:theorem} was proved by Duistermaat and H\"ormander  \cite{Duistermaat-Hormander:FIOs-2}, see also Èskin \cite{Eskin1970}. Other references related to the global boundedness of Fourier integral operators are presented in Section \ref{preliminaries}. For now, we emphasise that, in terms of the $L^p$-regularity, Theorem \ref{Tao:SSS:theorem} shows that Fourier integral operators lose $s=(d-1)|\frac{1}{2}-\frac{1}{p}|$ derivatives. In other words, on the scale of Sobolev spaces $L^p_s,$ a Fourier integral operator $T=T_{a(\cdot,t,\cdot)}^{\phi(\cdot,t,\cdot)}$ of order $m\in \mathbb{R},$ satisfies the time-dependent estimate
\begin{equation}\label{ref:Sogge}
    \Vert T f\Vert_{L^p_s(\mathbb{R}^d_x)}\lesssim_t   \Vert f\Vert_{L^p_{s+m+(d-1)|\frac{1}{2}-\frac{1}{p}|}(\mathbb{R}^d_x)},\,1<p<\infty.
\end{equation} It was first observed by Sogge \cite{Sogge1991}, in the context of the wave equation, that averaging in time on the left-hand side of \eqref{ref:Sogge} allows one to improve the loss of regularity. This observation led to the formulation of the (linear) {\it local smoothing conjecture} for the wave operator. If true, this conjecture would represent a remarkable breakthrough, as it formally implies several other major open problems in harmonic analysis, including the Bochner–Riesz, Fourier restriction, and Kakeya conjectures, see Tao \cite{Tao:BR:1999}.

In the setting of Fourier integral operators, the main assumption on the phase to formulate the linear local smoothing conjecture is the following {\it cinematic curvature condition}.
 \begin{definition}[Cinematic curvature condition]\label{CCC}
     We say that $\phi:=\phi(x,t,\xi)\in C^\infty(\mathbb{R}^d_x\times \mathbb{R}_t\times (\mathbb{R}^d\setminus\{0\}))$ satisfies the {\it cinematic curvature condition} on a compact set $\mathcal K\subset\mathbb{R}^d\times \mathbb{R}_t $ if the following conditions are satisfied:
 \begin{itemize}
     \item(i) the phase $\phi$ is non-degenerate, that is
     \begin{equation}
         \textnormal{rank}(\partial_{x,\eta}^2\phi(x,t,\eta))=d,\quad (x,t)\in \mathcal K,\quad \eta\in \mathbb{R}^d\setminus\{0\}.
     \end{equation}
     \item(ii) Consider the Gauss mapping $G:\mathcal K\times (\mathbb{R}^d\setminus\{0\})\rightarrow \mathbb{S}^d,$ where $$G(x,t,\eta)=\frac{G_0(x,t,\eta)}{|G_0(x,t,\eta)|},$$ and 
     $$G_0(x,t,\eta):=\wedge_{1\leq j\leq d}\partial_{\eta_j}\partial_{x,t}\phi(x,t,\eta).$$ The following rank condition 
     \begin{equation}
         \textnormal{rank}(\partial_{\eta\eta}^2\langle \partial_{x,t}\phi(x,t,\eta),G(z,\eta_0) \rangle|_{\eta=\eta_0})=d-1,
     \end{equation}holds for every $(x,t,\eta_0)\in \mathcal K\times \mathbb{R}^d,$ $\eta_0\neq 0.$
 \end{itemize}
 \end{definition} Adapted to Fourier integral operators, the linear  {\it smoothing conjecture} can be stated as follows.  In order to formulate it we use the following notation for its critical index:
\begin{equation}
   \overline{p}_d:=\frac{2(d+1)}{d-1}, \textnormal{ if } \, d \textnormal{ is odd, and, }    \overline{p}_d:=\frac{2(d+2)}{d},  \textnormal{ if }\, d \textnormal{ is even}.
\end{equation}
\begin{conjecture}[Linear smoothing conjecture for FIOs]\label{LSC} Let $d\geq 2.$ Let $T_a^\phi $ be a Fourier integral operator of order $m\in \mathbb{R},$ and with non-degenerate phase function satisfying the {\it cinematic curvature condition}, see Definition \ref{CCC}. For any $p\geq \overline{p}_d,$ and $\sigma<\frac{1}{p},$ we have that 
\begin{equation}
    \Vert T_a^\phi f \Vert_{L^p(\mathbb{R}^d\times \mathcal{I})}\lesssim \Vert f\Vert_{L^p_{s}(\mathbb{R}^d)},\, s=m+(d-1)\left(\frac{1}{2}-\frac{1}{p_i}\right)-\sigma.
\end{equation} 
\end{conjecture}
\begin{remark}
    Examples constructed by Bourgain \cite{Bourgain1991,Bourgain1995} can be used to identify Fourier integral operators that satisfy the cinematic curvature condition but for which local smoothing fails for all $\sigma < \frac{1}{p},$ whenever $p < \overline{p}_d$. Compared with Theorem \ref{Tao:SSS:theorem},  Conjecture \ref{LSC} claims a gain of regularity of order $\sigma<\frac{1}{p}.$ We illustrate some values of the critical index $\overline{p}_d$ in Table \ref{tab:critical-pd}.
\begin{table}[h]
\centering
\renewcommand{\arraystretch}{1.3} 
\setlength{\tabcolsep}{6pt}      
\begin{tabular}{c|ccccccccccccccccccc}
$d$ & 2 & 3 & 4 & 5 & 6 & 7 & 8 & 9 & 10 & 11 & 12 & 13 & 14 & 15 & 16 & 17 & 18 & 19 & 20 \\ \hline
$\overline{p}_d$ & $4$ & $4$ & $3$ & $3$ & $\frac{8}{3}$ & $\frac{8}{3}$ & $\frac{5}{2}$ & $\frac{5}{2}$ & $\frac{12}{5}$ & $\frac{12}{5}$ & $\frac{7}{3}$ & $\frac{7}{3}$ & $\frac{16}{7}$ & $\frac{16}{7}$ & $\frac{9}{4}$ & $\frac{9}{4}$ & $\frac{20}{9}$ & $\frac{20}{9}$ & $\frac{11}{5}$
\end{tabular}
\caption{Critical index $\overline{p}_d$ for dimensions $d = 2$ to $20$.}
\label{tab:critical-pd}
\end{table} 
\end{remark} 
Let us introduce the required notation to formulate the bilinear counterpart to Conjecture \ref{LSC} by adopting the notion of bilinear Fourier integral operators as studied e.g. by Grafakos and Peloso \cite{Grafakos:Peloso:BFIO}, and for the form of the operators formulated here, see also Rodríguez-López, Rule, and  Staubach  \cite{Staubach2021}. 

Let $a\in C^\infty(\mathbb{R}^d_x\times \mathbb{R}_t\times \mathbb{R}^d_\xi\times \mathbb{R}^d_\eta )$ be a {\it bilinear symbol}. We say that $a$ has  order $m\in \mathbb{R},$ if it satisfies the following estimates
\begin{equation}
    |\partial_{x,t}^\beta \partial_\xi^{\alpha_1}\partial_\eta^{\alpha_2}a(x,t,\xi,\eta)|\lesssim_{\beta,\alpha_1,\alpha_2} (1+|\xi|+|\eta|)^{m-|\alpha_1|-|\alpha_2|}.
\end{equation}The Kohn-Nirenberg bilinear class $S^m_{1,0}:=S^m_{1,0}(\mathbb{R}^d_x\times \mathbb{R}_t\times \mathbb{R}^d_\xi\times \mathbb{R}^d_\eta)$  is the set formed by those bilinear symbols of order $m.$

 Let us consider the {\it bilinear Fourier integral operator }
 \begin{align}
     T^{\phi_1,\phi_2}_a(f,g)(x,t)=\smallint_{\mathbb{R}^d}\smallint_{\mathbb{R}^d}e^{2\pi i(\phi_1(x,t,\xi)+\phi_2(x,t,\eta)) }a(x,t,\xi,\eta)\widehat{f}(\xi)\widehat{g}(\eta)d\xi d\eta,
 \end{align} where the symbol $a$ is compactly supported in $(x,t)\in \mathbb{R}^d\times \mathbb{R},$ and each phase function $\phi_j(x,t,\xi)\in C^\infty(\mathbb{R}^d_x\times \mathbb{R}_t\times (\mathbb{R}^d\setminus\{0\}))$ is positively homogeneous of order $1$ in $\xi\in \mathbb{R}^d\setminus\{0\},$ and they satisfy the {\it cinematic curvature condition}, see Definition \ref{CCC}.

 Here, we will denote by $\mathcal{I}=[a,b]$ an arbitrary non-empty finite interval and then $-\infty<a<b<\infty.$ To reduce a little bit our analysis, since the symbol $a:=a(x,t,\xi,\eta)$ is assumed to be compactly supported in $(x,t),$ and without loss of generality we will choose a compact set $K\subset \mathbb{R}^d,$ such that,
 \begin{equation}
     (x,t)-\textnormal{supp}(a):=\{(\tilde x,\tilde t)\in \mathbb{R}^d\times \mathbb{R}: (\tilde x,\tilde t)\in \textnormal{supp}(a)\}\subset  K\times \mathcal{I},
 \end{equation} increasing the length of $\mathcal{I}$ if needed. It becomes clear that 
 $$ \Vert T^{\phi_1,\phi_2}_a(f,g) \Vert_{L^{p}(\mathbb{R}^d_x\times \mathbb{R}_t )}= \Vert T^{\phi_1,\phi_2}_a(f,g) \Vert_{L^{p}( K\times \mathcal{I} )}.$$

This fact will not be used in our formulation of the bilinear smoothing conjecture but will be crucial in our further analysis. 
Based on the linear smoothing conjecture for Fourier integral operators, it is natural to suppose that the following result holds in any dimension $d\geq 2$.
 \begin{conjecture}[Bilinear smoothing conjecture for FIOs]\label{Conjecture:bilinear} Let $d\geq 2.$  For any $p_1,p_2\geq \overline{p}_d,$  let $\sigma_1<\frac{1}{p_1},\,\sigma_2<\frac{1}{p_2}.$ Let $T^{\phi_1,\phi_2}_a$ be a bilinear Fourier integral operator, with  non-degenerate phase functions satisfying the {\it cinematic curvature condition}. Let $a\in S^m_{1,0}$ be its symbol, having  order $m<0,$ and being compactly supported in $(x,t)\in \mathbb{R}^d\times \mathbb{R}.$ Then 
 \begin{equation}  \Vert T^{\phi_1,\phi_2}_a(f,g) \Vert_{L^{p}(\mathbb{R}^d_x\times \mathbb{R}_t )}\lesssim \Vert f\Vert_{L^{p_1}_{s_1}(\mathbb{R}^d)} \Vert g\Vert_{L^{p_2}_{s_2}(\mathbb{R}^d)},
 \end{equation} where
 \begin{equation}\label{p:p1:p2}
\frac{1}{p}=\frac{1}{p_1}+\frac{1}{p_2},\quad   s_i=m_i+(d-1)\left(\frac{1}{2}-\frac{1}{p_i}\right)-\sigma_i,\quad m=m_1+m_2,\quad m_1\leq0, m_2<0.
 \end{equation}
     
 \end{conjecture}
\begin{remark}Observe that in \eqref{p:p1:p2} one has that
$
   \boxed{ p\geq \frac{\overline{p}_d}{2}}.
$ Note that there is no loss of generality in assuming that $m<0$ in Conjecture \ref{LSC}, see Remark \ref{Rem:linear:conj:m}. Moreover, 
    note that if $a(x,t,\xi,\eta)$ has compact support in $\xi,$ there is no restriction in assuming that $m_2<0$ has arbitrarily large absolute value.   
    Conversely, if $a(x,t,\xi,\eta)$ has compact support in $\eta,$ there is no restriction in assuming that $m_1<0$ has arbitrarily large absolute value.
\end{remark}

Analysing the relationship between the two conjectures above is the main goal of this work. Indeed, we prove the following result.
 \begin{theorem}\label{main:theorem}
The linear smoothing Conjecture \ref{LSC} implies the bilinear smoothing Conjecture \ref{Conjecture:bilinear}.
 \end{theorem}
The proof of Theorem \ref{main:theorem} will be deduced from a more qualitative version of it, see Theorem \ref{main:theorem:2}, which in particular will prove that any progress in the linear smoothing conjecture for FIOs implies a counterpart for bilinear FIOs in any dimension $d\geq 2.$
 
 The local smoothing conjecture has been verified in dimension $d=2$ by Guth, Wang, and Zhang \cite{GuthWangZhang} for solutions of the wave equation. Moreover, since the linear smoothing conjecture for FIOs has been proved in dimension $d=2$ by Gao, Liu, Miao, and Xi \cite{GaoLiuMiaoXi2023}, and for odd $d$ by Beltran, Hickman, and Sogge \cite{BeltranHickmanSogge}, their results, combined with Theorem \ref{main:theorem}, imply:
  
\begin{corollary}\label{corollary:d:2} For $d=2$ the bilinear local smoothing conjecture holds. In other words, for any $p_1,p_2\geq 4,$  let $\sigma_1<\frac{1}{p_1},\,\sigma_2<\frac{1}{p_2}.$ Let $a\in  S^m_{1,0}$ be a symbol of  order $m<0,$ being compactly supported in $(x,t)\in \mathbb{R}^2\times \mathbb{R}.$ Then 
 \begin{equation}  \Vert T^{\phi_1,\phi_2}_a(f,g) \Vert_{L^{p}(\mathbb{R}^2_x\times \mathbb{R}_t )}\lesssim \Vert f\Vert_{L^{p_1}_{s_1}(\mathbb{R}^2)} \Vert g\Vert_{L^{p_2}_{s_2}(\mathbb{R}^2)},
 \end{equation} where
 \begin{equation}
\frac{1}{p}=\frac{1}{p_1}+\frac{1}{p_2},\quad   s_i=m_i+\frac{1}{2}-\frac{1}{p_i}-\sigma_i,\quad m=m_1+m_2,\quad m_1\leq0, m_2<0.
 \end{equation}On the other hand, if $d$ is odd, the bilinear local smoothing conjecture holds. This means that for any $$p_1,p_2\geq \overline{p}_d:= \frac{2(d+1)}{d-1},$$  and for every $\sigma_1<\frac{1}{p_1},\,\sigma_2<\frac{1}{p_2},$ an operator $T^{\phi_1,\phi_2}_a$ having  order $m<0,$ with  non-degenerate phase functions satisfying the {\it cinematic curvature condition}, satisfies 
 \begin{equation}  \Vert T^{\phi_1,\phi_2}_a(f,g) \Vert_{L^{p}(\mathbb{R}^d_x\times \mathbb{R}_t )}\lesssim \Vert f\Vert_{L^{p_1}_{s_1}(\mathbb{R}^d)} \Vert g\Vert_{L^{p_2}_{s_2}(\mathbb{R}^d)},
 \end{equation} where
 \begin{equation}
\frac{1}{p}=\frac{1}{p_1}+\frac{1}{p_2},\quad   s_i=m_i+(d-1)\left(\frac{1}{2}-\frac{1}{p_i}\right)-\sigma_i,\quad m=m_1+m_2,\quad m_1\leq0, m_2<0.
 \end{equation}     
\end{corollary}
There is no sense of thinking about a calculus of bilinear FIOs, however, each bilinear FIO has encoded its own {\it calculus} of linear FIOs when investigating its inner structure \cite{Rodriguez:Lopez}.
We will adopt in our further analysis the structural properties introduced by Rodríguez-López, Rule, and  Staubach in \cite{Rodriguez:Lopez} for bilinear Fourier integral operators. These structural properties serve as a substitute for the calculus of Fourier integral operators due to H\"ormander \cite{Hormander1971Ac} in this setting. A natural question is to what extent a bilinear FIO is approximately the composition, (or the product) of two linear FIOs. Indeed, unless the bilinear symbol of an operator has compact support in one of the two frequency variables $(\xi,\eta)$ a bilinear FIO cannot be written as the {\it composition} of two linear FIOs.  Their method consists of making a careful continuous decomposition of the high frequencies of a bilinear FIO $T^{\phi_1,\phi_2}_{a}.$ Roughly explaining, it expresses the high frequency portion of the operator as a  sum of two operators of the form (see (46) of \cite[Page 18]{Rodriguez:Lopez})
\begin{equation}
    \tilde T^{\phi_1,\phi_2}_{\tilde a}(\mu(D)f,\mu(D)g)(x,s)=\smallint\limits_0^1\smallint\smallint T^{\phi_1}_{\nu_1^{t,u}}f(x,s) T^{\phi_2}_{\mu_1^{t,v}}g(x,s)\frac{m(t,x,s,U)}{(1+|U|^2)^N}du dv \frac{dt}{t},
\end{equation} where $|U|^2=|u|^2+|v|^2.$
The operators  $T^{\phi_1}_{\nu_1^{t,u}} T^{\phi_2}_{\mu_1^{t,v}}$ are approximations of linear Fourier
integral operators. In particular, the authors in \cite{Rodriguez:Lopez} have used this representation to establish a bilinear extension of the celebrated Seeger-Sogge-Stein theorem \cite{SSS}. 

In view of the linear smoothing property for FIOs proved in dimension $d=2$ by  Gao, Liu, Miao, and Xi \cite{GaoLiuMiaoXi2023}, it is of interest for us to ask whether or not the methods in \cite{Rodriguez:Lopez} can be applied to smoothing estimates in the bilinear context considering the heavy machinery used in their analysis from various parts of harmonic analysis; see, for instance the discussion in \cite[Page 3]{Rodriguez:Lopez}.

One of the main differences of bilinear smoothing estimates compared with the analysis for time-dependent boundedness theorems as in \cite{Rodriguez:Lopez,SSS}, came from the fact that we cannot make use of end-point estimates  at infinity involving the BMO space, or at $p=1,$ and then involving the Hardy space $H^1.$ In our case for the range of values $p_1,p_2\geq \overline{p}_d,$ we have to deal entirely with the $L^{p_1}_{s_1}\times L^{p_2}_{s_2}\rightarrow L^p$-estimate. Making use of square-maximal functions as in  \cite{Rodriguez:Lopez} we deal with the complexity of working with the direct $L^{p_1}_{s_1}\times L^{p_2}_{s_2}\rightarrow L^p$-boundedness theorem by applying in a certain portion of the continuous decomposition of the operator, a well known theorem for estimations of maximal functions, due to Bourgain \cite{Bourgain1985,Bourgain1986}, see Theorem \ref{Bourgain:l2:estimate}; in comparison with the previous methods in the literature this is the main novelty of this paper.

This paper is organised as follows. In Section \ref{Main:results} we present our main result (see Theorem \ref{main:theorem:2}) about smoothing estimates for bilinear FIOs. Section \ref{preliminaries} is dedicated to providing preliminaries about the inner structure of these operators. Section \ref{Proof:Section} is dedicated to the proof of Theorem \ref{main:theorem:2}.

\section{Smoothing Estimates for Bilinear FIOs}\label{Main:results}

We formulate our main result in Theorem \ref{main:theorem:2} presented below. First, we  introduce two necessary definitions. Throughout, we denote by $\mathscr{S}(\mathbb{R}^d)$ the Schwartz class on $\mathbb{R}^d.$ 

\begin{definition}
    Let $d\geq 2,$  and let $2\leq p^*\leq \infty.$ Let $T_a^\phi$ be a Fourier integral operator of order $m\in \mathbb{R},$ and with non-degenerate phase function satisfying the {\it cinematic curvature condition}, see Definition \ref{CCC}. If for a fixed Lebesgue index $p,$ with $p\geq p^*,$ there exists $s=s(p,m,d)$ such that 
\begin{equation}
 \forall f\in C^\infty_0(\mathbb{R}^d),\,   \Vert T_a^\phi f \Vert_{L^p(\mathbb{R}^d\times \mathcal{I})}\lesssim \Vert f\Vert_{L^p_{s}(\mathbb{R}^d)},
\end{equation} we say that $T$ satisfies the linear {\it local smoothing property} (LLSP) with parameters $(p^*,p,s).$
\end{definition}
\begin{remark}\label{remark:linear} Note that the local smoothing conjecture claims that any Fourier integral operator of order $m\in \mathbb{R}$ with phase satisfying the {\it cinematic curvature condition} should hold the {\it local smoothing property} (LLSP) with parameters $(\overline{p}_d,p,s),$  where
$$s=m+(d-1)\left(\frac{1}{2}-\frac{1}{p_i}\right)-\sigma,$$ and $\sigma<\frac{1}{p}.$
\end{remark}
\begin{definition}
    Let $d\geq 2,$  and let $2\leq p^*\leq \infty.$ Let $T_a^{\phi_1,\phi_2}$ be a bilinear Fourier integral operator of order $m\in \mathbb{R},$ and with non-degenerate phase functions $\phi_1,\phi_2,$ satisfying the {\it cinematic curvature condition}. If for a fixed pair $(p_1,p_2),$ with  $p_1,p_2\geq p^*,$ there exist $s_1=s(p_1,m,d)$ and $s_2=s(p_2,m,d)$  such that 
\begin{equation}
    \forall f,g\in C^\infty_0(\mathbb{R}^d),\,  \Vert T_a^{\phi_1,\phi_2}f \Vert_{L^p(\mathbb{R}^d\times \mathcal{I})}\lesssim \Vert f\Vert_{L^{p_1}_{s_1}(\mathbb{R}^d)} \Vert g\Vert_{L^{p_2}_{s_2}(\mathbb{R}^d)},
\end{equation}
with $\frac{1}{p}=\frac{1}{p_1}+\frac{1}{p_2},$
we say that $T_a^{\phi_1,\phi_2}$ satisfies the bilinear local smoothing property (BLSP) with parameters $(p^*,p, p_1,p_2,s_1,s_2).$
\end{definition}
\begin{remark}\label{remark:bilinear} Note that the bilinear local smoothing conjecture claims that any Fourier integral operator of order $m=m_1+m_2<0,$  $m_2<0,$ with phases satisfying the {\it cinematic curvature condition} should hold the {\it bilinear local smoothing property} with parameters $(\overline{p}_d,p,p_1,p_2,s_1,s_2),$  where
$$s_i=m_i+(d-1)\left(\frac{1}{2}-\frac{1}{p_i}\right)-\sigma_i,$$ where $\sigma_i$ is any number satisfying $\sigma_i<\frac{1}{p_i}.$    
\end{remark}

\begin{remark}\label{Rem:linear:conj:m}
    Note that there is no restriction in assuming that $m<0$ in Conjecture \ref{LSC}. Indeed, if $m_0\geq 0$ is the order of a Fourier Integral Operator $T^\phi_a,$ one can take $s_0$ such that $m=m_0-s_0<0,$ and to replace $f$ by $f_{s_0}=(1-\Delta)^{\frac{s_0}{2}}$ to get for the FIO $ \widetilde{T}=T(1-\Delta)^{-\frac{s_0}{2}}$ the smoothing property
\begin{equation*}
     \Vert \widetilde{T}f_{s_0} \Vert_{L^p(\mathbb{R}^d\times \mathcal{I})}\lesssim \Vert (1-\Delta)^{\frac{s_0}{2}} f\Vert_{L^p_{\tilde s}(\mathbb{R}^d)},\, \tilde s=m+(d-1)\left(\frac{1}{2}-\frac{1}{p_i}\right)-\sigma.
\end{equation*}The previous estimate is equivalent to the following one
    \begin{equation*}
     \Vert T f \Vert_{L^p(\mathbb{R}^d\times \mathcal{I})}\lesssim \Vert f\Vert_{L^p_{\tilde s+s_0}(\mathbb{R}^d)},\, \tilde s+s_0=m_0+(d-1)\left(\frac{1}{2}-\frac{1}{p_i}\right)-\sigma.
\end{equation*} This fact and our evidence justify our choice of assuming that the order $m$ is negative in Conjecture \ref{Conjecture:bilinear}. 
Moreover, according to our methodology, we decompose the operator $T^{\phi_1,\phi_2}_a$ as follows
 \begin{align*}  T^{\phi_1,\phi_2}_a(f,g) &= T^{\phi_1,\phi_2}_a(\mu(D)f,\mu(D)g)\\
 &+ T^{\phi_1,\phi_2}_a(\mu(D)f,(1-\mu)(D)g)+ T^{\phi_1,\phi_2}_a((1-\mu(D))f,g).
 \end{align*} where $\mu$ is supported outside of a ball centred at the origin.  The second and the third terms are the low-frequency portions of the operator.  Curiously, for the analyisis of the low-frequencies we need the hypothesis $m_1\leq 0,$ while for the analysis of the high frequencies we need to assume that $m_2<0.$  
\end{remark}
\begin{theorem}\label{main:theorem:2}
    Let $d\geq 2,$ and let $2\leq p^*\leq \infty.$ Let us consider the following hypothesis:
    \begin{itemize}
        \item[$(H1)$]  Assume that a pair of real numbers $(m_1,m_2),$ and a pair of Lebesgue indices $(p_1,p_2),$ are such that $p_i\geq p^*,$ and satisfying that any Fourier integral operator of order $ m_i$ satisfying the cinematic curvature condition, holds (LLSP) with parameters $(p^*,p_i,s_i)$ for some $s_i=s_i(p_i, m_i,d).$
    \end{itemize}
Then, if  $m:=m_1+m_2,$ $m_1\leq 0,$ and $m_2<0,$  every bilinear Fourier integral operator $T^{\phi_1,\phi_2}_a$ of order $m$ with  non-degenerate phase functions satisfying the {\it cinematic curvature condition,} satisfies (BLSP) with parameters $(p^*,p,p_1,p_2,s_1,s_2).$ In other words, we have that 
\begin{equation}
    \Vert T_a^{\phi_1,\phi_2}f \Vert_{L^p(\mathbb{R}^d\times \mathcal{I})}\lesssim \Vert f\Vert_{L^{p_1}_{s_1}(\mathbb{R}^d)} \Vert g\Vert_{L^{p_2}_{s_2}(\mathbb{R}^d)},
\end{equation}
where $\frac{1}{p}=\frac{1}{p_1}+\frac{1}{p_2},$ $1<p<\infty.$     
\end{theorem}
\begin{proof}[Proof of Theorem \ref{main:theorem}] The theorem follows from Theorem \ref{main:theorem:2} taking also into account 
    Remarks \ref{remark:linear} and \ref{remark:bilinear}.
\end{proof} 
\begin{proof}[Proof of Corollary  \ref{corollary:d:2}] The proof follows from Theorem \ref{main:theorem:2} and the solution of the linear smoothing conjecture in dimension $d=2$ due to Gao, Liu, Miao, and Xi \cite{GaoLiuMiaoXi2023}; also for $d$ odd Conjecture has been verified by Beltran, Hickman, and Sogge \cite{BeltranHickmanSogge}. For the numerology used here for this conjecture with $d$ odd we refer to \cite[Page 1926]{GaoLiuMiaoXi2023}. This fact combined with Theorem \ref{main:theorem:2} proves the statement in this case.  The proof is complete.   
\end{proof}

\section{Preliminaries}\label{preliminaries}

In this section, we present the structural properties of bilinear FIOs as introduced by Rodríguez-López, Rule, and  Staubach in \cite{Rodriguez:Lopez}. As mentioned before,  these structural properties replace the calculus of Fourier integral operators due to H\"ormander \cite{Hormander1971Ac} in this setting in the sense that we want to understand bilinear operators as twisted versions of compositions or products of linear FIOs. In other words, each bilinear FIO induces its own {\it calculus} of linear FIOs when investigating its inner structure \cite{Rodriguez:Lopez}. 

Standard notation as $\psi(D)$ or $\psi(tD)$ will be used later for the corresponding multipliers with symbols $\psi:=\psi(\eta),$ and $\psi(t\cdot):=\psi(t\eta),$ respectively.
\subsection{Structural Properties of Bilinear FIOs} First, we will review the structure of the bilinear operator $T^{\phi_1,\phi_2}_a,$ given by,
\begin{align}
     T^{\phi_1,\phi_2}_a(f,g)(x,t)=\smallint_{\mathbb{R}^d}\smallint_{\mathbb{R}^d}e^{2\pi i(\phi_1(x,t,\xi)+\phi_2(x,t,\eta)) }a(x,t,\xi,\eta)\widehat{f}(\xi)\widehat{g}(\eta)d\xi d\eta.
 \end{align} The first step will make a reduction of our analysis in low frequencies and in high frequencies. For this we follow the approach and notation due to  Rodríguez-López, Rule, and  Staubach in \cite[Sections 3-4]{Rodriguez:Lopez}. Some remarks are in order. 
\begin{remark}(c.f. \cite[Page 12]{Rodriguez:Lopez}).
    We will use the structure of the operator  $T^{\phi_1,\phi_2}_a$ to decompose it in a suitable continuous form. The phase functions $\phi_j$ are smooth, positively homogeneous of order one in the frequency variable and they are non-degenerate. By the mean value theorem
    \begin{equation}
     \forall (x,t,\xi)\in \mathcal{K}\times(\mathbb{R}^d\setminus\{0\})=K\times \mathcal{I} \times(\mathbb{R}^d\setminus\{0\}),\,   |\nabla_x\phi_j(x,t,\xi)|\lesssim |\xi|.
    \end{equation}In addition, the phases are non-degenerate and then one has the lower bound
    \begin{equation}
      \forall (x,t,\xi)\in \mathcal{K}\times(\mathbb{R}^d\setminus\{0\})=K\times \mathcal{I} \times(\mathbb{R}^d\setminus\{0\}),\,    |\xi| \lesssim  |\nabla_x\phi_j(x,t,\xi)|.
    \end{equation}Therefore, there exists $\lambda>0,$ such that on $K\times \mathcal{I} \times(\mathbb{R}^d\setminus\{0\})$ one has the estimate
    \begin{equation}
         \lambda|\xi| \leq |\nabla_x\phi_j(x,t,\xi)|\leq \frac{|\xi|}{\lambda}.
    \end{equation} We consider a smooth function $\mu$ such that $\mu\equiv 0,$ on $\{\xi\in \mathbb{R}^d:|\xi|\leq \frac{1}{4\lambda}\},$ and with $\mu\equiv 1$ on $\{\xi\in \mathbb{R}^d:|\xi|\geq \frac{1}{3\lambda}\}.$
\end{remark}
\begin{remark} We decompose the operator $T^{\phi_1,\phi_2}_a$ as follows
 \begin{align}  T^{\phi_1,\phi_2}_a(f,g) &= T^{\phi_1,\phi_2}_a(\mu(D)f,\mu(D)g)\\
 &+ T^{\phi_1,\phi_2}_a(\mu(D)f,(1-\mu)(D)g)+ T^{\phi_1,\phi_2}_a((1-\mu(D))f,g).
 \end{align} The second and the third term are the low-frequency portions of the operator.  That these parts satisfy the $L^{p_1}_{s_1}\times L^{p_2}_{s_2} \rightarrow L^p$ estimate requires an analysis based on understanding the low frequency portion of the operator as an almost composition of linear FIOs, see Subsection \ref{low:frequencies}. 
\end{remark}
\begin{remark}(c.f. \cite[Section 4]{Rodriguez:Lopez}). For the analysis of the high-frequency portion $T^{\phi_1,\phi_2}_a(\mu(D)f,\mu(D)g)$ one can introduce two smooth cut-off functions $\nu,\chi\in C^\infty_0(\mathbb{R}^d\times \mathbb{R}^d)$ such that $\chi(\xi,\eta)=1$ for $|\xi|+|\eta|\leq 1,$  $\chi(\xi,\eta)=0$ for $|\xi|+|\eta|\geq 2,$  $\nu(\xi,\eta)=0$ for $\lambda^2|\xi|\leq 16|\eta|,$ and $\nu(\xi,\eta)=1,$ for $64|\eta|\leq \lambda^2|\xi|.$ Let us define
\begin{equation}
    a_1:=(1-\chi)\nu a,\,\,a_2:=(1-\chi)(1-\nu)a.
\end{equation} We have that $a_1$ and $a_2$ are bilinear symbols of order $m,$ and 
 \begin{align}  T^{\phi_1,\phi_2}_a(\mu(D)f,\mu(D)g)=  T^{\phi_1,\phi_2}_{a_1}(\mu(D)f,\mu(D)g)+T^{\phi_1,\phi_2}_{a_2}(\mu(D)f,\mu(D)g).
 \end{align}   
 We have the representation (see (43) in \cite[Page 16]{Rodriguez:Lopez})
\begin{equation}\label{main:rep}
    T^{\phi_1,\phi_2}_{a_1}(\mu(D)f,\mu(D)g)(x,s)=\smallint\limits_0^1\smallint\smallint T^{\phi_1}_{\nu_1^{t,u}}f(x,s) T^{\phi_2}_{\mu_1^{t,v}}g(x,s)\frac{m(t,x,s,u,v)}{(1+|u|^2+|v|^2)^N}du dv \frac{dt}{t}.
\end{equation}
The operator $T^{\phi_1}_{\nu_1^{t,u}}$ has the following factorisation modulo an error term (see (54) in \cite[Page 20]{Rodriguez:Lopez}) 
\begin{equation}
    T^{\phi_1}_{\nu_1^{t,u}} =t^{-m_2} Q_t^u T_{\nu_{m_1}}^{\phi_1}+t^{-m_2}E_t^1,
\end{equation} where, 
\begin{itemize}
    \item $Q_t^u$ is a Fourier multiplier with symbol (see (54) in \cite[Page 21]{Rodriguez:Lopez}) 
    $$\xi \mapsto|t\xi|^{m_2}\widehat{\psi}(t\xi)e^{it u\cdot \xi}\in \mathscr{S}(\mathbb{R}^n).$$ The function $\psi\in \mathscr{S}(\mathbb{R}^n)$ is choosen to be smooth, whose Fourier transform is supported on the annulus $\{\xi\in \mathbb{R}^d:\frac{1}{2}\leq |\xi|\leq 2\}$ and such that 
    $$\forall t>0,\,\forall \xi\neq 0,\, \smallint\limits_0^\infty|\widehat{\psi}(t\xi)|^2\frac{dt}{t}=1.$$
    \item The operator $T_{\nu_{m_1}}^{\phi_1}$ is a Fourier integral operator of order $m_1$ with phase $\phi_1$ and with symbol (see (55) in \cite[Page 21]{Rodriguez:Lopez})
    \begin{equation}
        (x,s,\xi)\mapsto\nu_{m_1}(x,\xi)=\mu(\xi)\chi(x)|\nabla_x\phi_1(x,s,\xi)|^{m_1}\in S^{m_1}_{1,0}(\mathbb{R}^d_x\times \mathbb{R}_s\times \mathbb{R}^d_\xi).
    \end{equation}The operator $E_t^1$ has the same phase $\phi_1$ of $T_{\nu_{m_1}}^{\phi_1}$ and a symbol in the class  $S^{m_1-(\frac{1}{2}-\varepsilon)}_{1,0}(\mathbb{R}^d_x\times \mathbb{R}_s\times \mathbb{R}^d_\xi),$ with seminorms bounded by the product of $t^\varepsilon,$ $\varepsilon\in (0,\frac{1}{2}),$ with a polynomial expression in $|u|.$
\end{itemize}
On the other hand, the operator $T^{\phi_2}_{\mu_1^{t,v}}$ admits the following factorisation modulo an error term (see (56) in \cite[Page 21]{Rodriguez:Lopez})
$$t^{-m_2}T^{\phi_2}_{\mu_1^{t,v}}=P_t^vT^{\phi_2}_{\mu_1^t}+E_t^2,$$
where:
\begin{itemize}
    \item(i) the symbol of $P_t^v$ is given by 
    $$\eta\mapsto\widehat{\theta}(t\eta)e^{itv\cdot \eta} \in S^{0}_{1,0}( \mathbb{R}^d_\eta), $$
    \item(ii) and  $T^{\phi_2}_{\mu_1^t}$ is a Fourier integral operator of order $m_2<0,$ with phase function $\phi_2,$ and with symbol 
    $$(x,s,\eta)\mapsto \mu_1^t(x,s,\eta)=\mu(\eta)\chi(x)\widehat{\theta}_1(t\nabla_x\phi_2(x,s,\eta))t^{-m_2}\in S^{m_1}_{1,0}(\mathbb{R}^d_x\times \mathbb{R}_s\times \mathbb{R}^d_\eta),$$
    uniformly in $t\in (0,1].$  Here, $E_t^2$ is a Fourier integral operator with phase $\phi_2$    and a symbol in the class  $S^{m_2-(\frac{1}{2}-\varepsilon)}_{1,0}(\mathbb{R}^d_x\times \mathbb{R}_s\times \mathbb{R}^d_\eta),$ with seminorms that behaves polynomially in $|v|.$ 
\end{itemize}
\end{remark}
\subsection{Decomposition à la Coifman-Meyer} Some remarks are in order.
\begin{remark}
   The core of the proof of the structural properties for bilinear FIOs lies in  the representation in \eqref{main:rep} for $T^{\phi_1,\phi_2}_{a_1}$\footnote{A similar representation holds for $T^{\phi_1,\phi_2}_{a_2}.$ So, we will just analyse one of these two terms in the proof of our main theorem.} that is (see \cite[Page 21]{Rodriguez:Lopez}),
\begin{equation}\label{main:rep:2}
    T^{\phi_1,\phi_2}_{a_1}(\mu(D)f,\mu(D)g)(x,s)=\smallint\smallint \smallint\limits_0^1 T^{\phi_1}_{\nu_1^{t,u}}f(x,s) T^{\phi_2}_{\mu_1^{t,v}}g(x,s)\frac{m(t,x,s,u,v)}{(1+|u|^2+|v|^2)^N}du dv \frac{dt}{t}.
\end{equation}
About the $\frac{dt}{t}$-integral in \eqref{main:rep:2}, by writing $m(t,x,s):=m(t,x,s,u,v),$ one has that 
\begin{align*}
    \smallint\limits_0^1 T^{\phi_1}_{\nu_1^{t,u}}f(x,s) T^{\phi_2}_{\mu_1^{t,v}}g(x,s)m(t,x,s)\frac{dt}{t} &= I_1+I_2+I_3,
\end{align*} with
\begin{align*}
  I_1=  \smallint\limits_0^1 Q_t^u T^{\phi_1}_{\nu_{m_1}}f(x,s) P_t^v T^{\phi_2}_{\mu_1^t}g(x,s)m(t,x,s)\frac{dt}{t},
\end{align*}
\begin{align*}
  I_2=  \smallint\limits_0^1 Q_t^u T^{\phi_1}_{\nu_{m_1}} f(x,s) E_t^2g(x,s)m(t,x,s)\frac{dt}{t},
\end{align*} and 
\begin{align*}
 I_3=  \smallint\limits_0^1 E_t^1f(x,s)t^{-m_2} T^{\phi_2}_{\mu_1^{t,v}}g(x,s)m(t,x,s)\frac{dt}{t}.
\end{align*}
\end{remark}
\begin{remark}\label{splitting:I1}
    Note that for $m<0,$ and with $m_2<0$ we have the following representation (see \textbf{Case II} in \cite[Page 24]{Rodriguez:Lopez} and \cite[Page 26]{Rodriguez:Lopez})
\begin{align*}
    I_1=\sum_{j=1}^3\smallint_\kappa^\infty\smallint_0^1\chi_{(0,1)}(rt)(Q_t^u T_{\nu_{m_1} }^{\phi_1}f)(x,s)(P_t^v S_{r,t}^j g)(x,s)m(t,x,s)\frac{dt dr}{tr},
\end{align*} where $S_{r,t}^1=r^{m_2}P_{1,t}Q_{tr}T_\gamma^{\phi_2},$ with $T_\gamma^{\phi_2}$ being a Fourier integral operator of order $m_2$ with symbol
\begin{equation}  \gamma(x,s,\eta)=\chi(x)\mu(\eta)|\nabla_x\phi_2(x,s,\eta)|^{m_2}\in S^{m_2}_{1,0}(\mathbb{R}^d_x\times \mathbb{R}_s\times \mathbb{R}^d_\eta),
\end{equation} and $P_{1,t}$ is the Fourier multiplier with symbol $\widehat{\theta}_1(t\eta).$ The operators $S_{r,t}^2,S_{r,t}^3,$ have the form
\begin{align*}    S_{r,t}^2g(x,s)=s^{m_2}P_{1,t}\left(\smallint e^{i\phi_2(\cdot,s,\eta)} R^2_{tr}(\cdot,s,\eta)\widehat{g}(\eta)d\eta\right)(x),
\end{align*} and 
\begin{align*}    S_{r,t}^3g(x,s)=s^{m_2}\left(\smallint e^{i\phi_2(x,s,\eta)} R^1_{t}(x,s,\eta)\widehat{g}(\eta)d\eta\right),
\end{align*} where $R_t^1(x,s,\eta),R_t^1(x,s,\eta)\in S^{m_2-(\frac{1}{2}-\varepsilon)}_{1,0}(\mathbb{R}^d_x\times \mathbb{R}_s\times \mathbb{R}^d_\eta).$ Also, these symbols have seminorms that are bounded by $t^\varepsilon$  and by $(tr)^{\varepsilon},$ with $\varepsilon\in (0,1/2),$ respectively. 
\end{remark}
\begin{remark}\label{Pt:Qt}
    The operators $P_t^v$ and $Q_t^u,$ are uniformly $L^p(\mathbb{R}^d_x)$-bounded with operators norms growing polynomially in $|u|$ and $|v|,$ respectively; see e.g., \cite[Page 161]{SteinBook1993}.
\end{remark}
Having understood the {\it calculus} of Fourier integral operators inside a bilinear FIO we will start the analysis of low frequencies in the next section.

\section{Proof of Theorem \ref{main:theorem:2}}\label{Proof:Section}
In this section we prove our main Theorem \ref{main:theorem:2}. We split our analysis in two parts. First, we investigate the smoothing properties of the low frequencies of a bilinear FIO. This portion of the operator is approximately the composition of two linear Fourier Integral operators. As for the portion of the operator formed by the high frequencies, the idea is to express it in terms of {\it paraproducts} à la Coifman-Meyer \cite{CoifmanMeyer}.

\subsection{Analysis of the low frequencies}\label{low:frequencies}
For the analysis of low frequencies it suffices to prove that a Fourier integral operator
\begin{align}
     T^{\phi_1,\phi_2}_a(f,g)(x,t)=\smallint_{\mathbb{R}^d}\smallint_{\mathbb{R}^d}e^{2\pi i(\phi_1(x,t,\xi)+\phi_2(x,t,\eta)) }a(x,t,\xi,\eta)\widehat{f}(\xi)\widehat{g}(\eta)d\xi d\eta,
 \end{align} with a symbol $a:=a(x,t,\xi,\eta)$ compactly supported in $\xi$ or in $\eta$ satisfies the estimate
 \begin{align}\label{FIO:low:frequency}
      &\Vert T^{\phi_1,\phi_2}_{a_1}(f,g)\Vert_{L^{p}(  K\times \mathcal{I} )} 
   \lesssim \Vert f \Vert_{L^{p_1}_{s_1}(\mathbb{R}^d_x)} \Vert  g\Vert_{L^{p_2}_{s_2}(\mathbb{R}^d_x)}.
 \end{align} Here
 \begin{equation}
\frac{1}{p}=\frac{1}{p_1}+\frac{1}{p_2},\quad   s_i=m_i+(d-1)\left(\frac{1}{2}-\frac{1}{p_i}\right)-\sigma_i,\quad m=m_1+m_2,\quad m_1\leq 0,\,m_2<0.
 \end{equation} 
 Without loss of generality, let us assume that $f,g\in C^\infty_0(\mathbb{R}^n),$ and that  the symbol $a$ is compactly supported in $(x,t,\xi).$ The following remark will be useful in our further analysis.

 \begin{remark}\label{Rem}Assume that the parameters \( s_i \) are of the form
\[
s_i = m_i + \tilde{s}_i(p_i,d).
\]
Let \( N \) be an arbitrary integer.  Observe that if the symbol \( a(s,t,\xi,\eta) \) has compact support in \( \xi \), then there is no restriction in Theorem~\ref{main:theorem:2} in assuming that \( m_2 < 0 \) has arbitrarily large absolute value; in particular, we may assume that \( |m_2| > N \).  

Similarly, if \( a(s,t,\xi,\eta) \) has compact support in \( \eta \), there is no restriction in assuming that \( m_1 < 0 \) has arbitrarily large absolute value.  

To justify this claim, namely that it is enough to consider the conclusion of Theorem~\ref{main:theorem:2} when \( |m_2| \) is large, we consider the first case, where \( a(s,t,\xi,\eta) \) has compact support in \( \xi \).  

The idea is to consider \( m_2 \le 0 \) with small absolute value, and  reduce the proof to the case where \( m_2 \le 0 \) has large absolute value.

 Let $m=m_1+m_2,$ $m_2<0$ be the order of a bilinear FIO $T^{\phi_1,\phi_2}_a.$ One can take $s_0$ such that $m_2'=m_2-s_0<-N,$ and to replace $g$ by $g_{s_0}=(1-\Delta)^{\frac{s_0}{2}}g$ to get for the bilinear FIO 
    $$ \tilde T^{\phi_1,\phi_2}_a(\cdot,\cdot)=T^{\phi_1,\phi_2}_a(\cdot,(1-\Delta_x)^{-\frac{s_0}{2}}\cdot),$$ of order $m-s=(m_1-s_0)+m_2=m_1+m_2=m_1+(m_2-s_0)=m_1+m_2',$ that
 \begin{align*}
      &\Vert T^{\phi_1,\phi_2}_{a_1}(f,g)\Vert_{L^{p}(  K\times \mathcal{I} )}= \Vert \widetilde T^{\phi_1,\phi_2}_{a_1}(f,g_{s_0})\Vert_{L^{p}(  K\times \mathcal{I} )}\lesssim   \Vert f \Vert_{L^{p_1}_{s_1}(\mathbb{R}^d_x)} \Vert  g_{s_0}\Vert_{L^{p_2}_{s_2'}(\mathbb{R}^d_x)},
 \end{align*} where    
\begin{equation*}
      s_1=m_1+\tilde{s}_1(p_1,d),\, s_2'=m_2-s_0+\tilde{s}_2(p_1,d).
\end{equation*} Observe that 
\begin{equation}
    \,s_2:= s_2'+s_0=m_2+\tilde{s}_2(p_1,d)
\end{equation} does not depend on $s_0.$
Note also that in the case of the linear smoothing conjecture one has
$ 
    \tilde{s}_1(p_i,d)=(d-1)\left(\frac{1}{2}-\frac{1}{p_i}\right)-\sigma_i,\,\, \sigma_i<\frac{1}{p_i}.
$ 
It is clear that
$$\Vert f \Vert_{L^{p_1}_{s_1}(\mathbb{R}^d_x)} \Vert  g_{s_0}\Vert_{L^{p_2}_{s_2'}(\mathbb{R}^d_x)}= \Vert f \Vert_{L^{p_1}_{s_1}(\mathbb{R}^d_x)} \Vert  g\Vert_{L^{p_2}_{s_2'+s_0}(\mathbb{R}^d_x)}=\Vert f \Vert_{L^{p_1}_{s_1}(\mathbb{R}^d_x)} \Vert  g\Vert_{L^{p_2}_{s_2}(\mathbb{R}^d_x)}.$$
Observe that we have used that the order of $T^{\phi_1,\phi_2}_a(\cdot,(1-\Delta_x)^{-\frac{s_0}{2}}\cdot)$ is $m-s=m_1-s+m_2$ since its symbol is given by 
$$a(x,t,\xi,\eta)\langle\eta\rangle^{-s_0}\sim a(x,t,\xi,\eta)\langle\xi, \eta\rangle^{-s_0},$$
and is compactly supported in $\xi.$ The previous analysis shows that, under the condition of the symbol to be compactly supported in one of the two frequency variables, if the bilinear smoothing estimate holds for $m_2$ with $|m_2|$ large enough, then the same property holds with any $m_2<0.$
\end{remark}

We return our attention to the analysis of the low frequencies. For the proof of \eqref{FIO:low:frequency} let us combine the proof of the $L^2$-boundedness of pseudo-differential operators of order zero \cite[Page 234]{SteinBook1993} and the proof of Lemma 3.1 in \cite[Page 12]{Rodriguez:Lopez}.  Let $\psi$ be a smooth cutt-off function that is equal to $1$ on the support of $a$ in $\xi.$ The support of the symbol $a$ in $\xi$  will be denoted by $\mathfrak{S}.$  Then
 \begin{equation}
     T^{\phi_1,\phi_2}_{a_1}(f,g)(x,t)=\smallint e^{i\phi_1(x,t,\xi)}\mathfrak{a}_g(x,t,\xi)\widehat{f}(\xi)d\xi =T^{\phi_1}_{\mathfrak{a}_g}(f)(x,t),\, (x,t,\xi)\in K\times\mathcal{I}\times \mathfrak{S},
 \end{equation} where
 \begin{equation}
     \mathfrak{a}_g(x,t,\xi)=\smallint e^{i\phi_2(x,t,\eta)}\psi(\xi)a(x,t,\xi,\eta)\widehat{g}(\eta)d\eta=T_{ \psi(\xi)a(\cdot,\cdot,\xi,\cdot)}^{\phi_2}g(x,t).
 \end{equation} 
 Observe that the support of $\mathfrak{a}_g(x,t,\xi)$ in $(x,t)$ coincides with the support of $a$ in $(x,t)$ and then it is contained in $K\times \mathcal{I}.$ 
Using the Fourier inversion formula we have that
\begin{align*}
 T^{\phi_1}_{\mathfrak{a}_g}(f)(x,t) &=   \smallint e^{i\phi_1(x,t,\xi)}\mathfrak{a}_g(x,t,\xi)\zeta(\xi)\widehat{f}(\xi)d\xi \\
 &= C_d\smallint e^{i\phi_1(x,t,\xi)}\smallint\widehat{\mathfrak{a}}_g(x,t,\lambda)e^{i\xi\cdot \lambda}d\lambda \zeta(\xi)\widehat{f}(\xi)d\xi\\
 &= C_d \smallint  \widehat{\mathfrak{a}}_g(x,t,\lambda) \smallint e^{i\phi_1(x,t,\xi)}\zeta(\xi)\widehat{f(\cdot-\lambda)}(\xi)d\xi d\lambda \\
 &= C_d \smallint  \widehat{\mathfrak{a}}_g(x,t,\lambda) T_\zeta^{\phi_1} (f(\cdot-\lambda))(x,t) d\lambda .
\end{align*}
Note that $\widehat{\mathfrak{a}}_g(x,t,\lambda) $ has the same compact support in $x$ that ${\mathfrak{a}}_g(x,t,\lambda) .$ Let us consider a function $\chi\in C^\infty_0$ such that $\chi=1$ on $\textnormal{supp}_x(\widehat{\mathfrak{a}}_g).$ Then,
\begin{align*}
     T^{\phi_1}_{\mathfrak{a}_g}(f)(x,t) &= C_d \smallint  \widehat{\mathfrak{a}}_g(x,t,\lambda)\chi(x) T_\zeta^{\phi_1} (f(\cdot-\lambda))(x,t) d\lambda\\
     &= C_d \smallint  \widehat{\mathfrak{a}}_g(x,t,\lambda) T_{\chi\zeta}^{\phi_1} (f(\cdot-\lambda))(x,t) d\lambda.
\end{align*} H\"older inequality implies that 
\begin{align*}
    \Vert  T^{\phi_1}_{\mathfrak{a}_g}(f)(x,t) \Vert_{L^p(\mathbb{R}^d_x\times \mathbb{R}_t)} &\lesssim \smallint \Vert \widehat{\mathfrak{a}}_g(x,t,\lambda)\Vert_{L^{p_2}(\mathbb{R}^d_x\times \mathbb{R}_t)} \Vert T_{\chi\zeta}^{\phi_1} (f(\cdot-\lambda))(x,t)\Vert_{L^{p_1}(\mathbb{R}^d_x\times \mathbb{R}_t)}  d\lambda\\
    &\lesssim \smallint \Vert \widehat{\mathfrak{a}}_g(x,t,\lambda)\Vert_{L^{p_2}(\mathbb{R}^d_x\times \mathbb{R}_t)} \Vert f(\cdot-\lambda)\Vert_{L^{p_1}_{s_1}(\mathbb{R}^d_x)}  d\lambda.
\end{align*} Observe that, since $T_{\chi\zeta}^{\phi_1}$ has a symbol of order $-\infty,$  as the reviewer of this paper as remarked, the following estimate holds

\begin{align*}
     \Vert T_{\chi\zeta}^{\phi_1} (f(\cdot-\lambda))(x,t)\Vert_{L^{p_1}(\mathbb{R}^d_x\times \mathbb{R}_t)} \lesssim  \Vert f(\cdot-\lambda)\Vert_{L^{p_1}_{s_1}}=\Vert ((1-\Delta_x)^{\frac{s_1}{2}}f)(\cdot-\lambda)\Vert_{L^{p_1}}.
\end{align*}  Indeed, the operator $(1-\Delta_x)^{\frac{s_1}{2}}$ commutes with the translation operator $\tau_\lambda: f\mapsto f(\cdot-\lambda)$. Since the Lebesgue norms are also invariant by the action of the translation operator $\tau_\lambda,$ we have that
\begin{align*}
     \Vert T_{\chi\zeta}^{\phi_1} (f(\cdot-\lambda))(x,t)\Vert_{L^{p_1}(\mathbb{R}^d_x\times \mathbb{R}_t)} \lesssim \Vert f \Vert_{L^{p_1}_{s_1}}.
\end{align*}
In order to estimate the norm $\Vert \widehat{\mathfrak{a}}_g(x,t,\lambda)\Vert_{L^{p_2}(\mathbb{R}^d_x\times \mathbb{R}_t)},$ observe that 
\begin{align*}
     \widehat{\mathfrak{a}}_g(x,t,\lambda)=\smallint e^{i\lambda \xi} {\mathfrak{a}}_g(x,t,\lambda)dx=\langle\lambda\rangle^{-2N}  \smallint e^{i\lambda \xi} (I-\Delta_\xi)^{N}{\mathfrak{a}}_g(x,t,\xi)d\xi.
\end{align*}Then
\begin{align*}    
\Vert \widehat{\mathfrak{a}}_g(x,t,\lambda)\Vert_{L^{p_2}(\mathbb{R}^d_x\times \mathbb{R}_t)}\lesssim \langle\lambda\rangle^{-2N}  \smallint_\mathfrak{S}  \Vert (I-\Delta_\xi)^{N}{\mathfrak{a}}_g(x,t,\xi) \Vert_{L^{p_2}(\mathbb{R}^d_x\times \mathbb{R}_t)}d\xi,
\end{align*}where we have denoted the support of the symbol $a$ in $\xi$  by $\mathfrak{S}.$ 
We can compute 
\begin{align*}
    (I-\Delta_{\xi})^{N} \mathfrak{a}_g(x,t,\xi)
    &= (I-\Delta_{\xi})^{N}(\psi(\xi)T_{a(\cdot,\cdot,\xi,\cdot)}^{\phi_2}g(x,t))\\
    &= T_{\varkappa(\xi)}^{\phi_2} g(x,t).
\end{align*}
Note that the symbol of $ T_{\varkappa(\xi)}^{\phi_2}$ is given by $$\varkappa(\xi):= (1-\Delta_{\xi})^{N}[\psi(\xi)a(\cdot,\cdot,\xi,\cdot)],$$ and is a linear combination of symbols of order $\leq m=m_1+m_2\leq m_2,$ where we have used that $m_1\leq 0.$ Thus,  $\varkappa$ is a symbol of order $m_2$ and we have the smoothing estimate
\begin{align*}
   \Vert T^{\phi_2}_{\varkappa(\xi)} g(x,t)\Vert_{L^{p_2}} \lesssim \Vert g\Vert_{L^{p_2}_{s_2}}.
 \end{align*} 
 By taking $N>d/2,$ we have that 
 \begin{align*}
    \Vert  T^{\phi_1}_{\mathfrak{a}_g}(f)(x,t) \Vert_{L^p(\mathbb{R}^d_x\times \mathbb{R}_t)}
    &\lesssim \smallint \Vert \widehat{\mathfrak{a}}_g(x,t,\lambda)\Vert_{L^{p_2}(\mathbb{R}^d_x\times \mathbb{R}_t)} \Vert f\Vert_{L^{p_1}_{s_1}(\mathbb{R}^d_x)}  d\lambda.
\end{align*} 
 Summarising, we have proved that 
 \begin{align*}
    \Vert T^{\phi_1,\phi_2}_{a_1}(f,g)\Vert_{L^p(K\times \mathcal{I})} 
     &\lesssim \smallint\langle\lambda\rangle^{-2N}  \smallint_\mathfrak{S}  \Vert (I-\Delta_\xi)^{N}{\mathfrak{a}}_g(x,t,\xi) \Vert_{L^{p_2}(\mathbb{R}^d_x\times \mathbb{R}_t)}d\xi\Vert f\Vert_{L^{p_1}_{s_1}(\mathbb{R}^d_x)}  d\lambda\\
      &= \smallint\langle\lambda\rangle^{-2N}  \smallint_\mathfrak{S}  \Vert T_{\varkappa(\xi)}^{\phi_2} g\Vert_{L^{p_2}(\mathbb{R}^d_x\times \mathbb{R}_t)}\Vert f\Vert_{L^{p_1}_{s_1}(\mathbb{R}^d_x)}  d\lambda\\
      &\lesssim  \smallint\langle\lambda\rangle^{-2N} d\lambda \sup_{\xi\in \mathfrak{S}}\Vert T_{\varkappa(\xi)}\Vert_{\mathscr{B}(L^{p_2}_{s_2},L^{p_2})} |\mathfrak{S}|  \Vert g\Vert_{L^{p_2}_{s_2}(\mathbb{R}^d_x)}\Vert f\Vert_{L^{p_1}_{s_1}(\mathbb{R}^d_x)}\\
      &\lesssim_{K\times \mathcal{I}\times \mathfrak{S}} \Vert f \Vert_{L^{p_1}_{s_1}(\mathbb{R}^d_x)} \Vert  g\Vert_{L^{p_2}_{s_2}(\mathbb{R}^d_x)}.
 \end{align*} Thus, the previous estimate prove \eqref{FIO:low:frequency}.

\subsection{An estimate due to Bourgain} In our further analysis for high-frequencies we will make a continuous decomposition of the operator $T^{\phi_1,\phi_2}_{a_1}$ and we will estimate its norm using certain maximal functions. Of crucial interest for us will be the next estimate due to Bourgain.
 \begin{theorem}[Bourgain \cite{Bourgain1986}]\label{Bourgain:l2:estimate}
    Let $K\in L^1(\mathbb{R}^d)$ be such that its Fourier transform $\widehat{K}$ is differentiable. Define the following quantities
    \begin{equation}
        \alpha_j=\sup_{|\xi|\sim 2^{j}}|\widehat{K}(\xi)|,\quad \beta_j=\sup_{|\xi|\sim 2^{j}}|\langle\nabla \widehat{K}(\xi),\xi \rangle|.\footnote{Here, $\nabla K=(\partial_{x_i}K)_{1\leq i\leq d},$ denotes the gradient of $K,$ and the bracket $\langle x,y\rangle=\sum_{i=1}^{d}x_{i}y_{i},$ denotes the inner product on $\mathbb{R}^d.$}
    \end{equation}Then, we have the following estimate on the maximal operator associated to $K,$
    \begin{equation}\label{Sharp:estimate}
        \Vert \sup_{t>0}|f\ast K_t|  \Vert_{L^2(\mathbb{R}^d)}\leq C\Gamma(K)\Vert f \Vert_{L^2(\mathbb{R}^d)},
    \end{equation}
    for every $f\in \mathscr{S}(\mathbb{R}^d),$ where 
    \begin{equation}\label{Gamma:K}
       \Gamma(K)=\sum_{j\in \mathbb{Z}}\alpha_j^{\frac{1}{2}}(\alpha_j+\beta_j)^{\frac{1}{2}}. 
    \end{equation}
\end{theorem} There are several reasons making Theorem \ref{Bourgain:l2:estimate} interesting by itself. First, the constant $\Gamma(K)$ in \eqref{Gamma:K} is sharp in some sense, see \cite[Page 1475]{Bourgain1986}. On the other hand, it has been useful in the analysis of the $L^2$-theory of maximal functions, for convex bodies and for the sphere $\mathbb{S}^{n-1}\subset \mathbb{R}^{n}$, see e.g. the introduction of \cite{Cardona:IMRN:2024}.

\begin{remark}[Stein \cite{SteinBook1993}, Pages 27-29]\label{square:b:lp}
    We also recall that if $\psi\in \mathscr{S}(\mathbb{R}^d)$ is such that its Fourier transform $\widehat{\psi}$ is supported in an annulus isolating the origin (and then $\widehat{\psi}(0)=\smallint \psi(\varkappa)d\varkappa=0$), one has the following $L^p$-boundedness property:
    \begin{equation}
        \Vert \left(\smallint\limits_{\mathbb{R}^d_x}|u\ast \psi_t(x)|^2dx\right)^\frac{1}{2}\Vert_{L^p(\mathbb{R}^d_x)}\lesssim \Vert u\Vert_{L^p(\mathbb{R}^d_x)},\,\,\psi_t:=t^{-d}\psi(t^{-1}\cdot).
    \end{equation}This fact will be used later.
\end{remark}

\subsection{Analysis of the high frequencies} The analysis of the high frequencies will consists in proving that
\begin{equation}  \Vert T^{\phi_1,\phi_2}_a(\mu(D)f,\mu(D)g) \Vert_{L^{p}(\mathbb{R}^d_x\times \mathbb{R}_t )}\lesssim \Vert f\Vert_{L^{p_1}_{s_1}(\mathbb{R}^d)} \Vert g\Vert_{L^{p_2}_{s_2}(\mathbb{R}^d)},
 \end{equation} where
 \begin{equation}
\frac{1}{p}=\frac{1}{p_1}+\frac{1}{p_2},\quad   s_i=m_i+(d-1)\left(\frac{1}{2}-\frac{1}{p_i}\right)-\sigma_i,\quad m=m_1+m_2,\quad m_i<0.
 \end{equation} This is the remainig main step in the proof of Theorem \ref{main:theorem:2}.
\begin{proof}[Proof of Theorem \ref{main:theorem:2}] Since the analysis of the low frequencies has been carried out in Subsection \ref{low:frequencies}, according to the notation above, we need to prove the estimate
\begin{align*}
   & \Vert  {I}_1 \Vert_{L^{p}(\mathbb{R}^d_x\times \mathbb{R}_t )}+\|I_2\Vert_{L^{p}(\mathbb{R}^d_x\times \mathbb{R}_t )}+\|I_3\Vert_{L^{p}(\mathbb{R}^d_x\times \mathbb{R}_t )}\lesssim\Vert f\Vert_{L^{p_1}_{s_1}(\mathbb{R}_x^d)} \Vert g\Vert_{L^{p_2}_{s_2}(\mathbb{R}_x^d)}.
\end{align*} We recall that the first term $I_1$ can be decomposed as follows (see Remark \ref{splitting:I1}).
    \begin{align*}
    I_1=\sum_{j=1}^3\smallint_\kappa^\infty\smallint_0^1\chi_{(0,1)}(rt)(Q_t^u T_{\nu_{m_1} }^{\phi_1}f)(x,s)(P_t^v S_{r,t}^j g)(x,s)m(t,x,s)\frac{dt dr}{tr}.
\end{align*}
The operators $P_t^v$ and $Q_t^u,$ are uniformly $L^p(\mathbb{R}^d_x)$-bounded with operators norms growing polynomially in $|u|$ and $|v|,$ respectively (see Remark \ref{Pt:Qt}).
In consequence, the following $L^{p_2}(\mathbb{R}_x^d)$-estimates are satisfied:
\begin{itemize}
    \item $\Vert P_t^v S_{r,t}^3g(x,s)\Vert_{L^{p_2}(\mathbb{R}_x^d)}\lesssim_v\Vert S_{r,t}^3g(x,s) \Vert_{L^{p_2}(\mathbb{R}_x^d)},$
     \item $\Vert P_t^v S_{r,t}^2g(x,s)\Vert_{L^{p_2}(\mathbb{R}_x^d)}\lesssim_v\Vert S_{r,t}^2g(x,s) \Vert_{L^{p_2}(\mathbb{R}_x^d)},$
     \item $\Vert  Q_t^u T_{\nu_{m_1} }^{\phi_1}f(x,s)\Vert_{L^{p_1}(\mathbb{R}_x^d)}\lesssim_u \Vert  T_{\nu_{m_1} }^{\phi_1}f(x,s) \Vert_{L^{p_1}(\mathbb{R}_x^d)}.$
\end{itemize} As a consequence, the following $L^{p_2}(\mathbb{R}_x^d\times \mathbb{R}_s)$-estimates hold:
\begin{itemize}
    \item $\Vert P_t^v S_{r,t}^3g(x,s)\Vert_{L^{p_2}(\mathbb{R}_x^d\times \mathbb{R}_s)}\lesssim_v\Vert S_{r,t}^3g(x,s) \Vert_{L^{p_2}(\mathbb{R}_x^d\times \mathbb{R}_s)},$
     \item $\Vert P_t^v S_{r,t}^2g(x,s)\Vert_{L^{p_2}(\mathbb{R}_x^d\times \mathbb{R}_s)}\lesssim_v\Vert S_{r,t}^2g(x,s) \Vert_{L^{p_2}(\mathbb{R}_x^d\times \mathbb{R}_s)},$
     \item $\Vert  Q_t^u T_{\nu_{m_1} }^{\phi_1}f(x,s)\Vert_{L^{p_1}(\mathbb{R}_x^d\times \mathbb{R}_s)}\lesssim_u \Vert  T_{\nu_{m_1} }^{\phi_1}f(x,s) \Vert_{L^{p_1}(\mathbb{R}_x^d\times \mathbb{R}_s)}.$
\end{itemize} Since the operators $S_{r,t}^1,S_{r,t}^2$ and $S_{r,t}^3,$ are Fourier integral operators with phase function $\phi_2(x,s,\eta)$ and of order $\leq m_2,$ and the order of the Fourier integral operator $T_{\nu_{m_1} }^{\phi_1}$ equals $m_1,$ in view of the hypothesis $(H1)$ of Theorem \ref{main:theorem:2}, we have that
\begin{itemize}
    \item $\Vert S_{r,t}^3g(x,s) \Vert_{L^{p_2}(\mathbb{R}_x^d\times \mathbb{R}_s)}\lesssim r^{m_2}t^\varepsilon \Vert g\Vert_{L^{p_2}_{s_2}(\mathbb{R}_x^d)},$
     \item $\Vert S_{r,t}^2g(x,s) \Vert_{L^{p_2}(\mathbb{R}_x^d\times \mathbb{R}_s)} \lesssim r^{m_2+\varepsilon}t^\varepsilon \Vert g\Vert_{L^{p_2}_{s_2}(\mathbb{R}_x^d)},$
     \item $ \Vert  T_{\nu_{m_1} }^{\phi_1}f(x,s) \Vert_{L^{p_1}(\mathbb{R}_x^d\times \mathbb{R}_s)} \lesssim   \Vert f\Vert_{L^{p_1}_{s_1}(\mathbb{R}_x^d)}.$
\end{itemize}
In consequence we have proved that following properties hold for $\varepsilon\in (0,1/2),$
\begin{itemize}
    \item $ \Vert P_t^v S_{r,t}^3g(x,s)\Vert_{L^{p_2}(\mathbb{R}_x^d\times \mathbb{R}_s)}\lesssim r^{m_2}t^\varepsilon \Vert g\Vert_{L^{p_2}_{s_2}(\mathbb{R}_x^d)},$
    \item $ \Vert P_t^v S_{r,t}^2g(x,s)\Vert_{L^{p_2}(\mathbb{R}_x^d\times \mathbb{R}_s)}\lesssim  r^{m_2+\varepsilon}t^\varepsilon \Vert g\Vert_{L^{p_2}_{s_2}(\mathbb{R}_x^d)},$ and,
    \item $ \Vert Q_t^u T_{\nu_{m_1} }^{\phi_1}f(x,s)\Vert_{L^{p_1}(\mathbb{R}_x^d\times \mathbb{R}_s)}\lesssim   \Vert f\Vert_{L^{p_1}_{s_1}(\mathbb{R}_x^d)}.$
\end{itemize} Below $\Vert m\Vert_{L^\infty}:=\Vert m\Vert_{L^\infty_U}$ denotes the $L^\infty$-norm of $m$ as a function of $(t,x,s)$ and then it is bounded polinomially as a function of $U=1+|u|+|v|.$ Observe that the following estimates are valid,
\begin{align*}
 & \left\Vert \smallint_\kappa^\infty\smallint_0^1\chi_{(0,1)}(rt)(Q_t^u T_{\nu_{m_1} }^{\phi_1}f)(x,s)(P_t^v S_{r,t}^3 g)(x,s)m(t,x,s)\frac{dt dr}{tr}\right\Vert_{L^{p}(  K\times \mathcal{I} )} \\ 
 &\leq  \smallint_\kappa^\infty\smallint_0^{\frac{1}{r}}\Vert(Q_t^u T_{\nu_{m_1} }^{\phi_1}f)(x,s)(P_t^v S_{r,t}^3 g)(x,s)m(t,x,s)\Vert_{L^{p}(  K\times \mathcal{I} )} \frac{dt dr}{tr}\\
  &\lesssim  \smallint_\kappa^\infty\smallint_0^{\frac{1}{r}}\Vert(Q_t^u T_{\nu_{m_1} }^{\phi_1}f)(x,s)\Vert_{L^{p_1}(  K\times \mathcal{I} )}\Vert (P_t^v S_{r,t}^3 g)(x,s)\Vert_{L^{p_2}(  K\times \mathcal{I} )}\Vert m\Vert_{L^\infty} \frac{dt dr}{tr}\\
  &\lesssim \smallint_\kappa^\infty\smallint_0^{\frac{1}{r}} r^{m_2}t^\varepsilon \frac{dt dr}{tr} \Vert f\Vert_{L^{p_1}_{s_1}(\mathbb{R}_x^d)} \Vert g\Vert_{L^{p_2}_{s_2}(\mathbb{R}_x^d)}.
\end{align*} We have used that $\smallint_\kappa^\infty\smallint_0^{\frac{1}{r}} r^{m_2}t^\varepsilon \frac{dt dr}{tr}<\infty,$ since $m_2<0.$
Also, we have the estimations below,
\begin{align*}
 & \left\Vert \smallint_\kappa^\infty\smallint_0^1\chi_{(0,1)}(rt)(Q_t^u T_{\nu_{m_1} }^{\phi_1}f)(x,s)(P_t^v S_{r,t}^2 g)(x,s)m(t,x,s)\frac{dt dr}{tr}\right\Vert_{L^{p}(  K\times \mathcal{I} )} \\ 
 &\leq  \smallint_\kappa^\infty\smallint_0^{\frac{1}{r}}\Vert(Q_t^u T_{\nu_{m_1} }^{\phi_1}f)(x,s)(P_t^v S_{r,t}^2 g)(x,s)m(t,x,s)\Vert_{L^{p}(  K\times \mathcal{I} )} \frac{dt dr}{tr}\\
  &\lesssim  \smallint_\kappa^\infty\smallint_0^{\frac{1}{r}}\Vert(Q_t^u T_{\nu_{m_1} }^{\phi_1}f)(x,s)\Vert_{L^{p_1}(  K\times \mathcal{I} )}\Vert (P_t^v S_{r,t}^2 g)(x,s)\Vert_{L^{p_2}(  K\times \mathcal{I} )}\Vert m\Vert_{L^\infty} \frac{dt dr}{tr}\\
  &\lesssim \smallint_\kappa^\infty\smallint_0^{\frac{1}{r}} r^{m_2+\varepsilon}t^\varepsilon \frac{dt dr}{tr} \Vert f\Vert_{L^{p_1}_{s_1}(\mathbb{R}_x^d)} \Vert g\Vert_{L^{p_2}_{s_2}(\mathbb{R}_x^d)}.
\end{align*}
Now, since $S_{r,t}^1=r^{m_2}P_{1,t}Q_{tr}T_\gamma^{\phi_2},$ in order to finish the estimate of $I_1$ we need to deal with the norm
$$ \left\Vert\smallint_\kappa^\infty r^{m_2}\smallint_0^1\chi_{(0,1)}(rt)(Q_t^u T_{\nu_{m_1} }^{\phi_1}f)(x,s)(P_t^v [P_{1,t}Q_{tr}T_\gamma^{\phi_2}]g)(x,s)m(t,x,s)\frac{dt dr}{tr} \right\Vert_{L^{p}( K\times \mathcal{I})}.$$
To estimate the previous norm let us take an arbitrary $h\in L^{p'}( K\times \mathcal{I})$ assumed to be supported in $K\times \mathcal{I},$ when considering its extension by zero to $\mathbb{R}^d\times \mathbb{R}.$
It suffices to prove that
\begin{align*}
   &|\mathscr{I}|= \left|\smallint\limits_{  K\times \mathcal{I} }\smallint\limits_0^1 h(x,s)  (Q_t^u T_{\nu_{m_1} }^{\phi_1}f)(x,s)(\widetilde{P}_t^v [R_tT_\gamma^{\phi_2}]g)(x,s)m(t,x,s)\frac{dt}{t}dxds\right|\\
   &\lesssim \Vert h\Vert_{L^{p'}(  K\times \mathcal{I} )} \Vert f\Vert_{L^{p_1}_{s_1}(\mathbb{R}_x^d)} \Vert g\Vert_{L^{p_2}_{s_2}(\mathbb{R}_x^d)},
\end{align*}
where 
\begin{align*}
    R_t(G)(x)=\smallint\limits_\kappa^{\frac{1}{t}}r^{m_2}Q_{tr}(G)(x)\frac{dr}{r}, \textnormal{  and  } \quad\widetilde{P}_t^v :=P_t^v P_{1,t}.
\end{align*}
Let us write 
\begin{align*}
    &\mathscr{I}=\smallint\limits_{  K\times \mathcal{I} }\smallint\limits_0^1 h(x,s)  (Q_t^u T_{\nu_{m_1} }^{\phi_1}f)(x,s)(\widetilde{P}_t^v [R_tT_\gamma^{\phi_2}]g)(x,s)m(t,x,s)\frac{dt}{t}dxds\\
    &=I+II,
\end{align*}where $I$ and $II$ will be defined later but the idea will be transfer a little bit of the information of the multipliers $Q_t^u$ and $\widetilde{P}_t^u$ to the product $hm$. First, we assume without loss of generality that the $x$-support of $h$ is contained in the $x$-support of $m(t,x,s).$ We decompose
$$ \widetilde{P}_t^v={P}_t^{v,2}+{Q}_t^{v,2},$$
where the symbol of ${P}_t^{v,2}$ is given by $\widehat{\theta}_2(t\xi)e^{it\xi\cdot v}.$ The support of $\widehat{\theta}_2$ is choosen small enough  to assure that, by choosing a suitable radial function $\widehat{\psi}_2$ supported in an annulus, the identity
$$(Q_t^u T_{\nu_{m_1} }^{\phi_1}f)(x,s)({P}_t^{v,2} [R_tT_\gamma^{\phi_2}]g)(x,s)=\widetilde{Q}_t (Q_t^u T_{\nu_{m_1} }^{\phi_1}f)(x,s)({P}_t^{v,2} [R_tT_\gamma^{\phi_2}]g)(x,s),$$ holds.
We have denoted by $\widetilde{Q}_t$ the multiplier with symbol $\widehat{\psi}(t\xi).$ In a similar way, we can choose a radial function $\widehat{\theta}_3,$ supported in a ball such that 
$$(Q_t^u T_{\nu_{m_1} }^{\phi_1}f)(x,s)({Q}_t^{v,2} [R_tT_\gamma^{\phi_2}]g)(x,s)=\widetilde{P}_t (Q_t^u T_{\nu_{m_1} }^{\phi_1}f)(x,s)({Q}_t^{v,2} [R_tT_\gamma^{\phi_2}]g)(x,s),$$ with $\widetilde{P}_t$ denoting the multiplier with symbol $\widehat{\theta}_3(t\xi).$ We have that

 \begin{align*}
   &|\mathscr{I}|= \left|\smallint\limits_{  K\times \mathcal{I} }\smallint\limits_0^1 h(x,s)  (Q_t^u T_{\nu_{m_1} }^{\phi_1}f)(x,s)(\widetilde{P}_t^v [R_tT_\gamma^{\phi_2}]g)(x,s)m(t,x,s)\frac{dt}{t}dxds\right|\\
   &\leq\left|\smallint\limits_{  K\times \mathcal{I} }\smallint\limits_0^1 h(x,s)  (Q_t^u T_{\nu_{m_1} }^{\phi_1}f)(x,s)({P}_t^{v,2} [R_tT_\gamma^{\phi_2}]g)(x,s)m(t,x,s)\frac{dt}{t}dxds\right|\\
   &+\left|\smallint\limits_{  K\times \mathcal{I} }\smallint\limits_0^1 h(x,s)  (Q_t^u T_{\nu_{m_1} }^{\phi_1}f)(x,s)({Q}_t^{v,2}[R_tT_\gamma^{\phi_2}]g)(x,s)m(t,x,s)\frac{dt}{t}dxds\right|\\
    &=\left|\smallint\limits_{  K\times \mathcal{I} }\smallint\limits_0^1 \widetilde{Q}_t(Q_t^u T_{\nu_{m_1} }^{\phi_1}f)(x,s)({P}_t^{v,2} [R_tT_\gamma^{\phi_2}]g)(x,s) h(x,s) m(t,x,s)\frac{dt}{t}dxds\right|\\
   &+\left|\smallint\limits_{  K\times \mathcal{I} }\smallint\limits_0^1  \widetilde{P}_t(Q_t^u T_{\nu_{m_1} }^{\phi_1}f)(x,s)({Q}_t^{v,2}[R_tT_\gamma^{\phi_2}]g)(x,s) h(x,s)m(t,x,s)\frac{dt}{t}dxds\right|\\
   &=\left|\smallint\limits_{  K\times \mathcal{I} }\smallint\limits_0^1 (Q_t^u T_{\nu_{m_1} }^{\phi_1}f)(x,s)({P}_t^{v,2} [R_tT_\gamma^{\phi_2}]g)(x,s) \widetilde{Q}_t(h(\cdot,s) m(t,\cdot,s))(x)\frac{dt}{t}dxds\right|\\
   &+\left|\smallint\limits_{  K\times \mathcal{I} }\smallint\limits_0^1  (Q_t^u T_{\nu_{m_1} }^{\phi_1}f)(x,s)({Q}_t^{v,2}[R_tT_\gamma^{\phi_2}]g)(x,s) \widetilde{P}_t( h(\cdot,s)m(t,\cdot,s))(x)\frac{dt}{t}dxds\right|.
\end{align*}Let us define 
\begin{align}
    I=\smallint\limits_{  K\times \mathcal{I} }\smallint\limits_0^1 (Q_t^u T_{\nu_{m_1} }^{\phi_1}f)(x,s)({P}_t^{v,2} [R_tT_\gamma^{\phi_2}]g)(x,s) \widetilde{Q}_t(h(\cdot,s) m(t,\cdot,s))(x)\frac{dt}{t}dxds,
\end{align} and 
\begin{align*}
    II=\smallint\limits_{  K\times \mathcal{I} }\smallint\limits_0^1 (Q_t^u T_{\nu_{m_1} }^{\phi_1}f)(x,s)({P}_t^{v,2} [R_tT_\gamma^{\phi_2}]g)(x,s) \widetilde{P}_t(h(\cdot,s) m(t,\cdot,s))(x)\frac{dt}{t}dxds.
\end{align*}Now, consider the multiplication operator $$M_m G\equiv M_m^{t,s}G:=m(t,x,s)G.$$ The idea will be to make the estimation of $I$ and of $II$ by carring out an analysis over the commutators $[\widetilde{P}_t,M_m]$ and $[\widetilde{Q}_t,M_m].$ 

Starting with $I$ we have that
\begin{align}
    I &=\smallint\limits_{  K\times \mathcal{I} }\smallint\limits_0^1 (Q_t^u T_{\nu_{m_1} }^{\phi_1}f)(x,s)({P}_t^{v,2} [R_tT_\gamma^{\phi_2}]g)(x,s) m(t,x,s) (\widetilde{Q}_t h(\cdot,s) )(x)\frac{dt}{t}dxds\\
    &+ \smallint\limits_{  K\times \mathcal{I} }\smallint\limits_0^1 (Q_t^u T_{\nu_{m_1} }^{\phi_1}f)(x,s)({P}_t^{v,2} [R_tT_\gamma^{\phi_2}]g)(x,s) ([\widetilde{Q}_t,M_m]h(\cdot,s))(x) \frac{dt}{t}dxds\\
    &=I_4+I_5,
\end{align}with 
$$ I_5= \smallint\limits_{  K\times \mathcal{I} }\smallint\limits_0^1 (Q_t^u T_{\nu_{m_1} }^{\phi_1}f)(x,s)({P}_t^{v,2} [R_tT_\gamma^{\phi_2}]g)(x,s) [\widetilde{Q}_t,M_m](h(\cdot,s) )(x)\frac{dt}{t}dxds.$$ Similarly,
\begin{align*}
    II&=\smallint\limits_{  K\times \mathcal{I} }\smallint\limits_0^1 (Q_t^u T_{\nu_{m_1} }^{\phi_1}f)(x,s)({Q}_t^{v,2} [R_tT_\gamma^{\phi_2}]g)(x,s)m(t,x,s) (\widetilde{P}_th(\cdot,s))(x)\frac{dt}{t}dxds\\
    &+ \smallint\limits_{  K\times \mathcal{I} }\smallint\limits_0^1 (Q_t^u T_{\nu_{m_1} }^{\phi_1}f)(x,s)({Q}_t^{v,2} [R_tT_\gamma^{\phi_2}]g)(x,s) [\widetilde{P}_t,M_m](h(\cdot,s))(x)\frac{dt}{t}dxds\\
    &=I_6+I_7,
\end{align*} 
with 
$$I_7= \smallint\limits_{  K\times \mathcal{I} }\smallint\limits_0^1 (Q_t^u T_{\nu_{m_1} }^{\phi_1}f)(x,s)({Q}_t^{v,2} [R_tT_\gamma^{\phi_2}]g)(x,s) [\widetilde{P}_t,M_m]h(\cdot,s) (x)\frac{dt}{t}dxds.$$
The terms involving the commutators, namely, $I_5$ and $I_7$ can be estimated due to the $L^{p'}$-boundedness of $ [\widetilde{P}_t,M_m]$ and of $ [\widetilde{Q}_t,M_m].$ Indeed, they satisfy the boundedness properties
\begin{align*}
   \Vert [\widetilde{P}_t,M_m] h(\cdot,s)\Vert_{L^{p'}(\mathbb{R}^d_x)}\lesssim  t\Vert  h(\cdot,s)\Vert_{L^{p'}(\mathbb{R}^2_x)},\,  \Vert [\widetilde{Q}_t,M_m] h(\cdot,s)\Vert_{L^{p'}(\mathbb{R}^d_x)}\lesssim  t\Vert h(\cdot,s)\Vert_{L^{p'}(\mathbb{R}^d_x)}.
\end{align*} In view of the hypothesis $(H1)$ in Theorem \ref{main:theorem:2}  we have that 
\begin{align*}
    \Vert {P}_t^{v,2} [R_tT_\gamma^{\phi_2}]g\Vert_{L^{p_2}(  K\times \mathcal{I} )}\lesssim C(v)\Vert T_\gamma^{\phi_2}g\Vert_{ L^{p_2}( \mathbb{R}^d\times \mathbb{R} ) } \lesssim C(v)\Vert g\Vert_{ L^{p_2}_{s_2}(\mathbb{R}^d_x) },
\end{align*} and also the estimate
\begin{align*}
    \Vert {Q}_t^{v,2} [R_tT_\gamma^{\phi_2}]g\Vert_{L^{p_2}(  K\times \mathcal{I} )}\lesssim C(v)\Vert T_\gamma^{\phi_2}g\Vert_{ L^{p_2}(  \mathbb{R}^d\times \mathbb{R}) } \lesssim C(v)\Vert g\Vert_{ L^{p_2}_{s_2}(\mathbb{R}^d_x) }.
\end{align*} Thus, the analysis above implies that
$$|I_5|+|I_7|\lesssim C(u,v,\varepsilon)\Vert f\Vert_{L^{p_1}_{s_1}(\mathbb{R}_x^d)}\Vert g\Vert_{L^{p_2}_{s_2}(\mathbb{R}_x^d)}\Vert h\Vert_{L^{p'}(\mathbb{R}^d_x\times\mathbb{R}_s)}.$$
In order to finish our analysis of $\mathscr{I}$ let us estimate the terms $I_4$ and $I_6.$ By using Cauchy-Schwarz inequality we have that 
\begin{align*}
    &|I_4|\\
    &=\left|\smallint\limits_{  K\times \mathcal{I} }\smallint\limits_0^1 (Q_t^u T_{\nu_{m_1} }^{\phi_1}f)(x,s)({P}_t^{v,2} [R_tT_\gamma^{\phi_2}]g)(x,s) m(t,x,s) (\widetilde{Q}_t h(\cdot,s) )(x)\frac{dt}{t}dxds\right|\\
    &= \left|\smallint\limits_{\mathcal{I}}\smallint\limits_{K}\smallint\limits_0^1 (Q_t^u T_{\nu_{m_1} }^{\phi_1}f)(x,s)({P}_t^{v,2} [R_tT_\gamma^{\phi_2}]g)(x,s) m(t,x,s) (\widetilde{Q}_t h(\cdot,s) )(x)\frac{dt}{t}dx ds\right| \\
    &\leq \left(\smallint\limits_{\mathcal{I}}\smallint\limits_{K}\smallint\limits_0^1 |(Q_t^u T_{\nu_{m_1} }^{\phi_1}f)(x,s)|^{2}\frac{dt}{t}dx ds\right)^\frac{1}{2} \\
    &\times \left(\smallint\limits_{\mathcal{I}}\smallint\limits_{K}\smallint\limits_0^1 |({P}_t^{v,2} [R_tT_\gamma^{\phi_2}]g)(x,s) m(t,x,s) (\widetilde{Q}_t h(\cdot,s) )(x)|^{2}\frac{dt}{t}dx ds\right)^\frac{1}{2}. 
\end{align*}
For the first term, Minkowski integral inequality and Remark \ref{square:b:lp} imply that
\begin{align*}
&\left(\smallint\limits_{\mathcal{I}}\smallint\limits_{K}\smallint\limits_0^1 |(Q_t^u T_{\nu_{m_1} }^{\phi_1}f)(x,s)|^{2}\frac{dt}{t}dx ds\right)^\frac{1}{2}\\
&\leq \smallint\limits_{\mathcal{I}}\smallint\limits_{K} \left(\smallint\limits_0^1 |(Q_t^u T_{\nu_{m_1} }^{\phi_1}f)(x,s)|^{2}\frac{dt}{t}\right)^\frac{1}{2} dx ds\\
& \leq \smallint\limits_{\mathcal{I}} \left(\smallint\limits_{K} \left(\smallint\limits_0^1 |(Q_t^u T_{\nu_{m_1} }^{\phi_1}f)(x,s)|^{2}\frac{dt}{t}\right)^\frac{p_1}{2} dx \right)^{\frac{1}{p_1}}\left(\smallint_{K}dx\right)^{\frac{1}{p_1'}}ds\\
&\lesssim  \smallint\limits_{\mathcal{I}} \left\Vert \left(\smallint\limits_0^1 |(Q_t^u T_{\nu_{m_1} }^{\phi_1}f)(x,s)|^{2}\frac{dt}{t}\right)^\frac{1}{2}  \right\Vert_{L^{p_1}(K)}ds\\
&\lesssim  \smallint\limits_{\mathcal{I}} \Vert T_{\nu_{m_1} }^{\phi_1}f  \Vert_{L^{p_1}}ds\\
&\lesssim  \left(\smallint\limits_{\mathcal{I}} \Vert T_{\nu_{m_1} }^{\phi_1}f  \Vert_{L^{p_1}}^{p_1}ds\right)^\frac{1}{p_1}\left(\smallint_{\mathcal{I}}ds\right)^{\frac{1}{p_1'}}ds.
\end{align*}
For the second factor, observe that 
\begin{align*}    &\left(\smallint\limits_{\mathcal{I}}\smallint\limits_{K}\smallint\limits_0^1 |({P}_t^{v,2} [R_tT_\gamma^{\phi_2}]g)(x,s) m(t,x,s) (\widetilde{Q}_t h(\cdot,s) )(x)|^{2}\frac{dt}{t}dx ds\right)^\frac{1}{2}\\
&\leq \Vert m\Vert_{L^\infty} \left(\smallint\limits_{\mathcal{I}}\smallint\limits_{K}\smallint\limits_0^1 |({P}_t^{v,2} [R_tT_\gamma^{\phi_2}]g)(x,s) (\widetilde{Q}_t h(\cdot,s) )(x)|^{2}\frac{dt}{t}dx ds\right)^\frac{1}{2}.
\end{align*}
Let 
$$A_t= ({P}_t^{v,2} [R_tT_\gamma^{\phi_2}]g)(x,s),\,B_t=(\widetilde{Q}_t h(\cdot,s) )(x).$$
We can estimate the integral involving these two terms using repeatedly H\"older's inequality after using Minkowski inequality. Indeed
\begin{align*}    &\left(\smallint\limits_{\mathcal{I}}\smallint\limits_{K}\smallint\limits_0^1 |A_t B_t|^{2}\frac{dt}{t}dx ds\right)^\frac{1}{2}\\
&\leq \smallint\limits_{\mathcal{I}}\smallint\limits_{K} \left(\smallint\limits_0^1 |A_t B_t|^{2}\frac{dt}{t}\right)^\frac{1}{2} dx ds\\
& \leq \smallint\limits_{\mathcal{I}}\smallint\limits_{K} \sup_{0<t<1} |A_t| \left(\smallint\limits_0^1 | B_t|^{2}\frac{dt}{t}\right)^\frac{1}{2} dx ds\\
&\leq \left( \smallint\limits_{\mathcal{I}}\smallint\limits_{K} \sup_{0<t<1} |A_t|^pdxds \right)^{\frac{1}{p}} \left( \smallint\limits_{\mathcal{I}}\smallint\limits_{K} \left(\smallint\limits_0^1 | B_t|^{2}\frac{dt}{t}\right)^\frac{p'}{2}dxds \right)^{\frac{1}{p'}}\\
&\leq \left(\left( \smallint\limits_{\mathcal{I}}\smallint\limits_{K} \sup_{0<t<1} |A_t|^\frac{pp_2}{p}dxds \right)^{\frac{p}{p_2}} \left(\smallint_{K\times \mathcal{I}}dxds\right)^{\frac{1}{(p_2/p)'}} \right)^{\frac{1}{p}} \left\Vert\left(\smallint\limits_0^1 | B_t|^{2}\frac{dt}{t}\right)^\frac{1}{2}\right\Vert_{L^{p'}}\\
&\lesssim  \Vert\sup_{0<t<1} |A_t|\Vert_{L^{p_2}}\left\Vert\left(\smallint\limits_0^1 | B_t|^{2}\frac{dt}{t}\right)^\frac{1}{2}\right\Vert_{L^{p'}}.
\end{align*}Note that the use of H\"older' inequality before is justified by the fact that $\frac{p_2}{p}\geq 1,$ since $\frac{1}{p}=\frac{1}{p_1}+\frac{1}{p_2}\geq \frac{1}{p_2}.$
In view of the hypothesis $(H1)$ in Theorem \ref{main:theorem:2} and of Remark \ref{square:b:lp},  we have that 
\begin{align*}
    \Vert\sup_{0<t<1} |A_t|\Vert_{L^{p_2}}\left\Vert\left(\smallint\limits_0^1 | B_t|^{2}\frac{dt}{t}\right)^\frac{1}{2}\right\Vert_{L^{p'}} &\lesssim C(v)\Vert T_\gamma^{\phi_2}g\Vert_{ L^{p_2}(  \mathbb{R}^d\times \mathbb{R} ) } \Vert h\Vert_{L^{p'}(\mathbb{R}^d\times \mathbb{R} )}\\
    &\lesssim C(v)\Vert g\Vert_{ L^{p_2}_{s_2}(\mathbb{R}^d_x) } \Vert h\Vert_{L^{p'}(\mathbb{R}^d_x\times\mathbb{R}_s)}.
\end{align*}
Observe that we have used that the estimate
\begin{align}\label{sup:Bourgain}
    \Vert\sup_{0<t<1} |A_t|\Vert_{L^{p_2}}\lesssim C(v)\Vert T_\gamma^{\phi_2}g\Vert_{ L^{p_2}(  \mathbb{R}^d\times \mathbb{R} ) },
\end{align} holds. Indeed, in order to prove this fact observe that 
$$ A_t= ({P}_t^{v,2} [R_tT_\gamma^{\phi_2}]g)(x,s)= R_t({P}_t^{v,2} [T_\gamma^{\phi_2}]g)(x,s).$$
Now, the action of the multiplier $R_t{P}_t^{v,2}$ on $G=G_s=T_\gamma^{\phi_2}g(\cdot,s),$ is given by 
\begin{align*}
    R_t({P}_t^{v,2} G)(x) &=\smallint_{\kappa}^{\frac{1}{t}}s^{m_2}Q_{ts}{P}_t^{v,2} G(x)\frac{ds}{s}\\
     &=\smallint_{\kappa}^{\frac{1}{t}}s^{m_2}\smallint e^{ix\cdot \xi}\widehat{\psi}(ts\xi)\widehat{\theta}_2(t\xi)e^{it\xi\cdot v} \widehat{G}(\xi)d\xi\frac{ds}{s}.
\end{align*} Due to the fact that the supports of $\widehat{\psi}$ and of  $\widehat{\theta}_2$ are contained in small annuli, so, on the support of these functions $|\xi|\asymp \frac{1}{ts}\asymp \frac{1}{t},$ from where we deduce that $s\asymp 1,$ and the integral in $s$ is supported on an interval around $s=1$ and on an annulus isolating $s=0.$ So, we have that 
\begin{align*}
     |R_t({P}_t^{v,2} G)(x)| &\lesssim \smallint\limits_{s\asymp 1}s^{m_2}\left|\smallint e^{ix\cdot \xi}\widehat{\psi}(ts\xi)\widehat{\theta}_2(t\xi)e^{it\xi\cdot v} \widehat{G}(\xi)d\xi\right|\frac{ds}{s}\\
      &\lesssim \smallint\limits_{s\asymp 1}\left|\smallint e^{ix\cdot \xi}\widehat{\psi}(ts\xi)\widehat{\theta}_2(t\xi)e^{it\xi\cdot v} \widehat{G}(\xi)d\xi\right| ds\\
      &= \smallint\limits_{s\asymp 1}\left|\widehat{\psi}(ts D)P_t^{v,2} G(x) \right| ds.
\end{align*} Let $\theta_2^v$ be the kernel with Fourier transform $\widehat{\theta}_2^v(\cdot),$ i.e., it is  the symbol of $P_t^{v,2},$ that is, $\widehat{\theta}_2^v(t\xi):=\widehat{\theta}_2(t\xi)e^{it\xi\cdot v}.$ Then
\begin{align*}
    \smallint\limits_{s\asymp 1}\left|\widehat{\psi}(ts D)P_t^{v,2} G(x) \right| ds
    &= \smallint\limits_{s\asymp 1}\left|G\ast (s^{-n}t^{-n}(\psi(s^{-1}\cdot)\ast (\theta_2^v)(t^{-1}\cdot)))(x) \right| ds\\
     &= \smallint\limits_{s\asymp 1}\left|\smallint G(x-y) s^{-n}t^{-n}(\psi(s^{-1}\cdot)\ast (\theta_2^v))(t^{-1}y))dy \right| ds\\
     &\asymp \smallint\limits_{s\asymp 1} \left|\smallint G(x-y) t^{-n}(\psi(s^{-1}\cdot)\ast (\theta_2^v))(t^{-1}y))dy \right| ds\\
     &\leq \Vert G\Vert_{L^\infty}  \smallint\limits_{s\asymp 1} t^{-n}\left|\smallint (\psi(s^{-1}\cdot)\ast (\theta_2^v))(t^{-1}y))dy \right| ds\\
     &\leq \Vert G\Vert_{L^\infty}  \smallint\limits_{s\asymp 1} \left|\smallint (\psi(s^{-1}\cdot)\ast (\theta_2^v)(y))dy \right| ds.
\end{align*}In consequence
\begin{align}\label{L:infty:Bourgain}
 \Vert  \sup_{0<t<1}  |R_t({P}_t^{v,2} G)\Vert_{L^\infty}\lesssim  \Vert G\Vert_{L^\infty}.
\end{align}
Now, let us prove the estimate
\begin{equation}\label{L2:Bourgain}
   \Vert \sup_{0<t<1}  |R_t({P}_t^{v,2} G)\Vert_{L^2}\lesssim  \Vert G\Vert_{L^2}.
\end{equation}    
 For this we will use Bourgain's Lemma (see Theorem \ref{Bourgain:l2:estimate}). We proceed with the proof of \eqref{L2:Bourgain} in the following remark.
 \begin{remark}
     Note that
\begin{align*}
    R_t({P}_t^{v,2} G)(x) &= \smallint\limits_{s\asymp 1}s^{m_2}\smallint e^{ix\cdot \xi}\widehat{\psi}(ts\xi)\widehat{\theta}_2(t\xi)e^{it\xi\cdot v} \widehat{G}(\xi)d\xi\frac{ds}{s}\\
     &= \smallint\limits_{s\asymp 1}s^{m_2}\widehat{\psi}(ts D)P_t^{v,2} G(x) \frac{ds}{s}\\
     &= \smallint\limits_{s\asymp 1}s^{m_2}G\ast (s^{-n}t^{-n}(\psi(s^{-1}\cdot)\ast (\theta_2^v)(t^{-1}\cdot))(x) \frac{ds}{s}\\
      &= \left(G\ast \smallint\limits_{s\asymp 1}s^{m_2}s^{-n}t^{-n}(\psi(s^{-1}\cdot)\ast (\theta_2^v)(t^{-1}\cdot)) \frac{ds}{s}\right) (x).\\
\end{align*}We have that the kernel $K_t$ of $ R_t{P}_t^{v,2}$
is given by
\begin{equation}
    K_t(x)=\smallint\limits_{s\asymp 1}s^{m_2} (s^{-n}t^{-n}(\psi(s^{-1}\cdot)\ast (\theta_2^v)(t^{-1}x)))\frac{ds}{s}.
\end{equation}
Let us write
\begin{equation*}
    K(x)=\smallint\limits_{s\asymp 1}s^{m_2} s^{-n}(\psi(s^{-1}\cdot)\ast (\theta_2^v))(x)\frac{ds}{s}.
\end{equation*}
Its Fourier transform is 
\begin{equation}
    \widehat{K}(\xi)= \smallint\limits_{s\asymp 1}s^{m_2-n}\mathscr{F}(\psi(s^{-1}\cdot)\ast  (\theta_2^v))(\xi))\frac{ds}{s},
\end{equation} and 
\begin{align*}
  \alpha_j:=  \sup_{|\xi|\sim 2^j}| \widehat{K}(\xi)| &\leq \smallint\limits_{s\asymp 1}\Vert \mathscr{F}(\psi(s^{-1}\cdot)\ast  (\theta_2^v))  \Vert_{L^\infty} s^{m_2-n}\frac{ds}{s}\\
     &\lesssim \smallint\limits_{s\asymp 1}  \Vert \psi(s^{-1}\cdot)\ast  (\theta_2^v)  \Vert_{L^1}\frac{ds}{s}.
\end{align*}
This shows that $ \sup_j\sup_{|\xi|\sim 2^j}| \widehat{K}(\xi)|$ is bounded. However, due to the properties of the support of $\psi$ and $\theta_2^v$ we can imporve this bound as follows
\begin{align*}
    \alpha_j\lesssim   \sup_{|\xi|\sim 2^j} \smallint\limits_{s\asymp 1}  | \widehat{\psi}(s \xi)\widehat{\theta}_2^\nu(\xi)|  \frac{ds}{s}\lesssim  \sup_{|\xi|\sim 2^j} \min\{|\xi|,|\xi|^{-1}\}=2^{-|j|}.
\end{align*}
On the other hand
\begin{align*}
    \nabla\widehat{K}(\xi)= (\smallint\limits_{s\asymp 1}s^{m_2-n}\mathscr{F}((\psi(s^{-1}\cdot)\ast  (\theta_2^v )) ix_j)(\xi))\frac{ds}{s})_{1\leq j\leq d}.
\end{align*}Therefore, we have the estimate
\begin{align*}
     \sup_{|\xi|\sim 2^j}| \langle \nabla\widehat{K}(\xi),\xi\rangle|&\leq \sum_{r=1}^d \smallint\limits_{s\asymp 1}s^{m_2-n}|\mathscr{F}((\psi(s^{-1}\cdot)\ast  (\theta_2^v)) i x_r)(\xi))||\xi_r|\frac{ds}{s}.
\end{align*}Observe that, using integration by parts, we have that
\begin{align*}
    \mathscr{F}((\psi(s^{-1}\cdot)\ast  (\theta_2^v )i x_r)(\xi))\xi_r &=\smallint e^{ix\xi} (\psi(s^{-1}\cdot)\ast  (\theta_2^v ))(x)i x_r\xi_rdx\\
    &=\smallint \partial_{x_r}( e^{ix\xi} )(\psi(s^{-1}\cdot)\ast  (\theta_2^v )(x) x_rdx\\
     &=-\smallint  e^{ix\xi} \partial_{x_r}[(\psi(s^{-1}\cdot)\ast  (\theta_2^v )(x)] x_rdx.
\end{align*}We have proved that
\begin{align*}
    \sup_{|\xi|\sim 2^j}| \langle \nabla\widehat{K}(\xi),\xi\rangle| &\leq \sum_{r=1}^d \smallint\limits_{s\asymp 1}s^{m_2-n}\Vert\mathscr{F}(\partial_{x_r}[(\psi(s^{-1}\cdot)\ast  (\theta_2^v )(x)] x_r)\Vert_{L^\infty}\frac{ds}{s}\\
    &\leq \sum_{r=1}^d \smallint\limits_{s\asymp 1}  \Vert \partial_{x_r}[(\psi(s^{-1}\cdot)\ast  (\theta_2^v )(x)] x_r \Vert_{L^1}\frac{ds}{s}<\infty.
\end{align*} We have included the previous computation to emphasise the cancellation of the quantity $|\xi_r|\lesssim 2^j$ after using integration by parts which we found interesting. Nevertheless, one can improve this bound. Indeed, due to the properties of the support of  $\psi$ and $\theta_2^v$ we have that $$Z_r:= (\psi(s^{-1}\cdot)\ast  (\theta_2^v)) i x_r\in \mathscr{S}(\mathbb{R}^d),$$ and that
\begin{align*}
    \beta_j 
    &\lesssim \sup_{|\xi|\sim 2^j} \sum_{r=1}^d \smallint\limits_{s\asymp 1}|\mathscr{F}(Z_r)(\xi)||\xi_r|ds\\
    &\asymp \sup_{|\xi|\sim 2^j}  \sum_{r=1}^d \smallint\limits_{s\asymp 1}|\mathscr{F}(\partial_{x_r}Z_r)(\xi)|ds \lesssim  \sup_{|\xi|\sim 2^j} \min\{|\xi|,|\xi|^{-1}\}=2^{-|j|}.
\end{align*} In consequence
\begin{equation}
       \Gamma(K):=\sum_{j\in \mathbb{Z}}\alpha_j^{\frac{1}{2}}(\alpha_j+\beta_j)^{\frac{1}{2}}<\infty. 
    \end{equation} In view of Theorem \ref{Bourgain:l2:estimate} we have proved that $$G\mapsto S^*G:=  \sup_{0<t<1}  |R_t({P}_t^{v,2} G)|:L^2\rightarrow L^2$$ is bounded. 
 \end{remark} 
 So, interpolating the $L^p$-estimates \eqref{L:infty:Bourgain} and \eqref{L2:Bourgain} for the sub-linear operator $S^*,$ defined via,
 $$S^*G:=  \sup_{0<t<1}  |R_t({P}_t^{v,2} G)|,$$ in view of the Marcinkiewicz interpolation theorem
 we have proved that for any $2\leq \kappa\leq \infty,$ one has that 
\begin{align}\label{L:varkappa}
 \Vert  \sup_{0<t<1}  |R_t({P}_t^{v,2} G)|\Vert_{L^\kappa}\lesssim  C_\kappa(v)\Vert G\Vert_{L^\kappa},
\end{align} with a polynomial constant $C_\varkappa(v)$ in $v.$ With $\kappa=p_2$ we have proved \eqref{sup:Bourgain} once taked the $L^{p_2}$-norm in the $s$-variable on both sides of \eqref{L:varkappa}.
In consequence
$$|I_4|\lesssim \Vert h\Vert_{L^{p'}(  K\times \mathcal{I} )} \Vert f\Vert_{L^{p_1}_{s_1}(\mathbb{R}_x^d)} \Vert g\Vert_{L^{p_2}_{s_2}(\mathbb{R}_x^d)}, $$ since $h$ is assumed to be supported in $K\times I.$
Since the same analysis done before for $I_4$ also works for $I_6,$ we have that
$$|I_4|+|I_6|\lesssim \Vert h\Vert_{L^{p'}(  K\times \mathcal{I} )} \Vert f\Vert_{L^{p_1}_{s_1}(\mathbb{R}_x^d)} \Vert g\Vert_{L^{p_2}_{s_2}(\mathbb{R}_x^d)}. $$ We have completed the proof of the estimate
\begin{align*}
   &|\mathscr{I}|\lesssim \Vert h\Vert_{L^{p'}(  K\times \mathcal{I} )} \Vert f\Vert_{L^{p_1}_{s_1}(\mathbb{R}_x^d)} \Vert g\Vert_{L^{p_2}_{s_2}(\mathbb{R}_x^d)}.
\end{align*}
It follows that 
\begin{align*}
    \Vert I_1\Vert_{L^p(  K\times \mathcal{I} )} \lesssim_{u,v}  \Vert f\Vert_{L^{p_1}_{s_1}(\mathbb{R}_x^d)} \Vert g\Vert_{L^{p_2}_{s_2}(\mathbb{R}_x^d)}.
\end{align*} 
To estimate $I_2,$ note that the hypothesis $(H1)$ in Theorem \ref{main:theorem:2} implies that 
\begin{align*}
    \Vert Q_t^u T^{\phi_1}_{\nu_{m_1}} f(x,s) \Vert_{L^{p_1}(  K\times \mathcal{I} )}\lesssim_{u}   \Vert f \Vert_{L^{p_1}_{s_1}(\mathbb{R}^d_x)}.
\end{align*}Then H\"older's inequality gives
\begin{align*}
    \Vert I_2\Vert_{L^{p}(  K\times \mathcal{I} )} 
    &\leq  \smallint\limits_0^1 \Vert Q_t^u T^{\phi_1}_{\nu_{m_1}} f(x,s) E_t^2g(x,s)m(t,x,s)\Vert_{L^{p}(  K\times \mathcal{I} )}\frac{dt}{t}\\
    &\leq  \smallint\limits_0^1 \Vert Q_t^u T^{\phi_1}_{\nu_{m_1}} f(x,s) \Vert_{L^{p_1}(  K\times \mathcal{I} )}\Vert E_t^2g(x,s)m(t,x,s)\Vert_{L^{p_2}(  K\times \mathcal{I} )}\frac{dt}{t}\\
    &\leq  \smallint\limits_0^1 \Vert Q_t^u T^{\phi_1}_{\nu_{m_1}} f(x,s) \Vert_{L^{p_1}(  K\times \mathcal{I} )}\Vert E_t^2g(x,s)\Vert_{L^{p_2}(  K\times \mathcal{I} )}\Vert m\Vert_{L^\infty}\frac{dt}{t}\\
    &\leq  \smallint\limits_0^1 \Vert E_t^2g(x,s)\Vert_{L^{p_2}(  K\times \mathcal{I} )}\Vert m\Vert_{L^\infty}\frac{dt}{t} \Vert f \Vert_{L^{p_1}_{s_1}(\mathbb{R}^d_x)}\\
     &\lesssim_{u}  \smallint\limits_0^1 t^{\varepsilon-1}\Vert m\Vert_{L^\infty} dt \Vert f \Vert_{L^{p_1}_{s_1}(\mathbb{R}^d_x)} \Vert  g\Vert_{L^{p_2}_{s_2}(\mathbb{R}^d_x)}.
\end{align*} The same analysis above applies for $I_3.$ Remember that $\Vert m\Vert_{L^\infty}:=\Vert m\Vert_{L^\infty_U}$ denotes the $L^\infty$-norm of $m$ as a function of $(t,x,s)$ and then it is bounded polinomially as a function of $U=1+|u|+|v|.$ In consequence, taking $N$ large enough,
we have from \eqref{main:rep} that 
\begin{align*}
   &\Vert T^{\phi_1,\phi_2}_{a_1}(\mu(D)f,\mu(D)g)(x,s)\Vert_{L^{p}(  K\times \mathcal{I} )} \\
   &\leq \smallint\smallint \frac{\Vert m\Vert_{L^\infty_U}}{(1+|u|^2+|v|^2)^N}du dv \Vert f \Vert_{L^{p_1}_{s_1}(\mathbb{R}^d_x)} \Vert  g\Vert_{L^{p_2}_{s_2}(\mathbb{R}^d_x)}
   \lesssim_{N} \Vert f \Vert_{L^{p_1}_{s_1}(\mathbb{R}^d_x)} \Vert  g\Vert_{L^{p_2}_{s_2}(\mathbb{R}^d_x)}.
\end{align*} The proof of Theorem \ref{main:theorem:2} is complete. 
\end{proof}

\section{Appendix: Bibliographic Discussion}\label{preliminaries}

The study of the continuity properties for FIOs has attracted considerable attention due to their applications to hyperbolic PDEs. Initial inquiries into global $L^p$ boundedness results can be traced to pioneering works from the late 1970s and early 1980s \cite{Asada1978, Peral1980, Miyachi1980}. Subsequent research has significantly expanded these results to a wide classes of symbols \cite{Ruzhansky2006, Ruzhansky:Sugimoto, Coriasco:Ruzhansky, Coriasco:Toft2016}. A notable line of inquiry involves establishing global criteria for FIOs within H\"ormander symbol classes $S^m_{\rho,\delta}$, specifically by relaxing the traditional $\delta < \rho$ condition in the setting of pseudo-differential operators to allow broader parameter ranges $0 \leq \delta < 1$ and $0 < \rho \leq 1$ \cite{Dos:Santos2014, KenigStaubach, Staubach2021, Staubach, Rodriguez:Lopez2015, Rodriguez:Lopez}. This development was partly inspired by analogous results in the theory of pseudo-differential operators, which demonstrated boundedness without the $\delta < \rho$ restriction \cite{Del2006,Hounie}. For general pseudo-differential operators on graded Lie groups we refer to the work of the author with M. Ruzhansky and J. Delgado \cite{CardonaDelgadoRuzhansky:Lp:graded}. The question of global parametrisation for Lagrangian distributions has also been explored through alternative geometric and analytical approaches \cite{Laptev:Safarov:Vassiliev, Ruzhansky:CWI-book}. Additionally, global definitions of FIOs have been formulated in specialised settings, such as on compact Lie groups using the representation theory of these groups \cite{CardonaRuzhansky:MSJ:Japan}.

The systematic theory of FIOs with complex-phase functions originated from the work of Melin and Sj\"ostrand \cite{Melin:Sjostrand1975}, motivated by constructing parametrices for partial differential operators with complex principal symbols \cite{Melin:Sjostrand1976}. This framework has found significant application in several complex variables, where the singular structure of the Bergman kernel can be described via asymptotic expansions involving complex-phase FIOs \cite{Boutet:de:Monvel82}. The comprehensive analysis of the Bergman kernel itself represents another major area of study \cite{Fefferman1974, Fefferman1976}.

A distinct and highly active research direction concerns smoothing estimates for FIOs, particularly following the resolution of the decoupling conjecture by Bourgain and Demeter \cite{Bourgain:Demeter:2015}. The conjecture regarding the inherent smoothing properties of FIOs was formally posed by Sogge \cite{Sogge1991}. While a complete survey of this extensive field is beyond our scope, the literature includes influential early results \cite{Bourgain1991} and numerous recent advances that explore various aspects of this smoothing phenomenon, employing diverse methods from harmonic analysis, see  \cite{Minicozzi:Sogge, GuthWangZhang, BeltranHickmanSogge, GaoLiuMiaoXi2023} and the references therein.\\

\noindent\textbf{Conflict of interests statement - Data Availability Statements.}  The author states that there is no conflict of interest.  Data sharing does not apply to this article as no datasets were generated or
analysed during the current study.
\\

\noindent\textbf{Acknowledgement.}{The author would like to thank the anonymous referee for their valuable suggestions, which helped improve this manuscript.}

\bibliographystyle{amsplain}

\end{document}